\documentclass{article}

\usepackage[english]{babel}

\usepackage{amsmath, amsthm, amssymb, dsfont, subcaption, tikz, hyperref, enumitem, soul}

\usepackage{csquotes}

\newcommand{\mres}{\mathbin{\vrule height 1.6ex depth 0pt width
0.13ex\vrule height 0.13ex depth 0pt width 1.3ex}}
\newcommand{\closure}[2][3]{
  {}\mkern#1mu\overline{\mkern-#1mu#2}}
\newcommand{\ssubset}{\Subset}
\DeclareMathOperator\B{B}
\newcommand{\C}{\mathrm{C}}
\newcommand*{\coleq}{\mathrel{\vcenter{\baselineskip0.5ex\lineskiplimit0pt\hbox{\normalsize.}\hbox{\normalsize.}}}=}
\renewcommand{\d}{\mathrm{d}}
\newcommand{\dx}{\,\d x}
\renewcommand{\L}{\mathrm{L}}
\newcommand{\loc}{\mathrm{loc}}
\renewcommand{\P}{\mathrm{P}}
\newcommand{\R}{\mathds{R}}
\newcommand\sd{\,\triangle\,}
\newcommand{\cx}{\closure{x}}

\theoremstyle{plain}
\newtheorem{theo}{Theorem}[section]
\newtheorem{prop}[theo]{Proposition}
\newtheorem{lemma}[theo]{Lemma}

\newtheorem{definition}[theo]{Definition}

\theoremstyle{remark}
\newtheorem{rem}[theo]{Remark}

\newtheorem{con}[theo]{Construction}
\newtheorem*{notation}{Notation}

\numberwithin{equation}{section}

\title{The Optimal Hölder Exponent\\in Massari's Regularity Theorem}
\author{Thomas Schmidt\footnote{Fachbereich Mathematik, Universität Hamburg, Bundesstr.\@ 55, 20146 Hamburg, Germany. Email addresses: \href{mailto:thomas.schmidt.math@uni-hamburg.de}{thomas.schmidt.math@uni-hamburg.de}, \href{mailto:jule.schuett@uni-hamburg.de}{jule.schuett@uni-hamburg.de}.}\qquad\qquad Jule Helena Schütt${}^\ast$}
\date{March 24th, 2025}

\begin{document}

\maketitle

\begin{abstract}
\noindent We determine the optimal Hölder exponent in Massari's regularity theorem for sets with variational mean curvature in $\mathrm{L}^p$. In fact, we obtain regularity with improved exponents and at the same time provide sharp counterexamples.
\end{abstract}

\bigskip\bigskip


\tableofcontents

\section{Introduction}

This paper is concerned with a fine regularity issue for (local) minimizers of
Massari's functional
\[
  \mathcal{F}_H^U(F)\coleq\P(F,U)-\int_{F\cap U}H\dx\,,
\]
where the dimension $n\ge2$, an open set $U\subseteq\R^n$, and a function
$H\in\L^1(U)$ are given and $\P(F,U)$ stands for the perimeter of measurable sets
$F\subseteq\R^n$ in $U$. The study of $\mathcal{F}_H^U$ has its seeds in De
Giorgi's classical theory \cite{DeGiorgi6061} for the case $H\equiv0$, where every (non-singular)
minimizer $E$ of the perimeter $\mathcal{F}_0^U=\P(\,\cdot\,,U)$ is bounded by a
minimal surface $U\cap\partial E$, and in analogy one expects that, whenever $E$
locally minimizes $\mathcal{F}_H^U$ with general $H$, then $U\cap\partial E$
should be a prescribed-mean-curvature hypersurface in $\R^n$ with mean curvature
given by $H$.

The cornerstone results of the theory have been extended from $H\equiv0$ to
general $H$ by Massari \cite{massari1974esistenza,massari1975frontiere}. Indeed, he obtained --- as
a main advantage of the variational approach --- a basic existence
theorem for minimizers of $\mathcal{F}_H^U$ with a generalized
Dirichlet boundary condition at $\partial U$ and moreover established
partial $\C^{1,\alpha}$ regularity of local minimizers of
$\mathcal{F}_H^U$ up to a closed singular set of Hausdorff dimension
at most $n{-}8$ in case of $H\in\L^p(U)$, $p>n$.
While well-known examples demonstrate that both the assumption $p>n$ and the
dimension bound $n{-}8$ in the regularity theorem are sharp, to our
knowledge the optimality of the Hölder exponent $\alpha$ has not yet been
addressed. Indeed, while the original paper \cite{massari1975frontiere} provides
the exponent $\alpha=\frac14(1{-}\frac np)$, from the regularity theory of
Tamanini \cite{TamaniniBoundariesofCaccioppolisets, tamaninni1984regularity} for almost-minimizers of the perimeter one can directly read off the better exponent $\alpha=\frac12(1{-}\frac np)$ (compare also the introduction of \cite{AmbrosioPaolini1999}). Interestingly, though
Tamanini's results are optimal for almost-minimizers, for minimizers
of $\mathcal{F}_H^U$ with $H\in\L^p(U)$, $p>n$, we here
bring up the better, sharp, and apparently new exponent
\[
  \alpha_\mathrm{opt}(n,p)\coleq\frac p{p+1}\bigg(1-\frac np\bigg)=\frac{p-n}{p+1}\,.
\]
In fact, we establish partial $\C^{1,\alpha}$ regularity for all
$\alpha<\alpha_\mathrm{opt}(n,p)$, and at the
same time, for all $n\ge2$ and $n<p<\infty$, disprove $\C^{1,\alpha}$ regularity
for any $\alpha>\alpha_\mathrm{opt}(n,p)$ by counterexamples.
This
confirms the conjecture of \cite{massari1994variational} that
$\lim_{p\to\infty}\alpha_\mathrm{opt}(n,p)=1$ should hold. Moreover, our results
are optimal on the scale of Hölder spaces except for resolving the case of
the limit exponent $\alpha=\alpha_\mathrm{opt}(n,p)$. We expect that $\C^{1,\alpha}$
regularity extends to this limit exponent, but we leave a proof by somewhat
different methods for future treatment elsewhere.

In order to prove our regularity result it suffices to work on the
regular set where a-priori $\C^{1,\alpha}$ regularity with
\emph{some} $\alpha>0$ is available and we can use an iterative strategy
to gradually improve on $\alpha$.
Indeed, in each step we exploit the $\C^{1,\alpha}$ regularity at hand in order to improve on power-type decay
estimates for the deviation from minimality and deduce $\C^{1,\alpha}$
regularity with some larger $\alpha$ from Tamanini's results. Since the
resulting sequence of exponents converges from below to
$\alpha_\mathrm{opt}(n,p)$, this leads to the conclusion. Actually, the details
of the reasoning turn out to be somewhat technical, and we refer the reader to
Section \ref{Sec: Improvement HE} for a careful implementation.

We find it interesting to point out that, whenever a first-variation
equation is at hand and expresses that the mean curvature is given by an
$\L^q$ function, $q>n{-}1$, \emph{on the hypersurface}
$U\cap\partial E$, then it follows from $\L^q$ theory for the
linearized equation and the Morrey-Sobolev embedding that $\C^1$
solutions are automatically $\C^{1,\alpha}$ with the optimal exponent
$\alpha=1{-}\frac{n-1}q$. However, this framework is different from
ours and does not apply when considering $H\in\L^p(U)$
and thus allowing discontinuity of $H$ along $U\cap\partial E$ since
$H$ has no canonical restriction to the hypersurface $U\cap\partial E$ and in
general the second term of $\mathcal{F}_H^U$ is non-differentiable at the
minimizer $E$. The latter point may be even more visible when parametrizing
hypersurface portions as graphs and thus passing to the non-parametric
functional
\[
  \int_\Omega\sqrt{1{+}|\nabla w|^2}\,\d y-\int_\Omega\int_{-r}^{w(y)}H(y,s)\,\d s\,\d y
\]
for functions $w\in\C^1(\overline\Omega)$ over bounded open
$\Omega\subset\R^{n-1}$ with $r>\|w\|_{\C(\Omega)}$. Here,
we can differentiate in the dependent variable $w$ at a minimizer $f$
in the sense of
$\frac{\d }{\d  t}\big|_{t=f(y)}\int_0^tH(y,s)\,\d s=H(y,f(y))$ only if
$H$ is continuous in $s$ at the surface point $(y,f(y))$. All in all,
this means that, for minimizers of $\mathcal{F}_H^U$ with
$H\in\L^p(U)$, we do not have a first-variation equation at our disposal,
and hence our framework is a non-differentiable one in the general tradition of
\cite{giaquinta1983differentiability}. This goes with the observation that
our exponent $\alpha_\mathrm{opt}(n,p)$ lies in between the exponent
$\frac12(1{-}\frac np)$ available by direct variational considerations
and a better exponent of type $1{-}\frac np$ to be expected in a
differentiable situation. Moreover, for $p\to\infty$, when $H$
``approaches continuity'' and $\mathcal{F}_H^U$ ``approaches
differentiability'', $\alpha_\mathrm{opt}(n,p)$ asymptotically
approaches the better exponent $1{-}\frac np$. We remark that, on a general
level, this behavior is reminiscent of the optimal exponent
$\beta_\mathrm{opt}=\frac\gamma{2-\gamma}$ in the $\C^{1,\beta}$
regularity theory
\cite{phillips1983minimization,giaquinta1984sharp,hamburger2007optimal,schmidt2009simple,yang2013optimal} for
minimizers of non-parametric functionals with
$\C^{0,\gamma}$-dependence on the dependent variable $w$. However, this theory does not directly yield our exponent $\alpha_\mathrm{opt}(n,p)$. Additionally, in Section
\ref{Sec: Dim 2 regularity}  we consistently complete the picture described with
a minor observation, possibly well known to experts. Indeed, we record that in
case $H\in\L^\infty(U)$ one cannot write down the first-variation
equation of $\mathcal{F}_H^U$ but at least can formulate a closely related differential
inequality. Specifically for $n=2$, from this inequality one can easily read
off even $\C^{1,1}$ regularity of minimizers of $\mathcal{F}_H^U$ with $H\in\L^\infty(U)$ while for $n\ge3$ the counterexample \cite[Remark 3.6]{massari1994variational} shows that
$\C^{1,\alpha}$ regularity for all $\alpha<1$ is best possible.

Finally, let us briefly discuss our sharp counterexamples, which crucially
depend on the construction of suitable functions $H\in\L^p(U)$ such that
(a cut-off of) the $\C^{1,\alpha}$-subgraph
\[
  E=\left\{(\cx,x_n)\in\R^n\,:\,|\closure x|^{1+\alpha}<x_n\right\}
\]
locally minimizes $\mathcal{F}_H^U$ for some $U\ssubset \R^n$. We actually give two
constructions. The first one covers $n=2$ only (and actually works with an
odd variant of the subgraph $E$), but explicitly determines the function
$H$ with $H\in\L^p(U)$ whenever $p>2$, $\alpha>\alpha_\mathrm{opt}(2,p)$, as the
divergence of a unit vector field which suitably extends the outward unit normal
of $E$. The main ingredient is a lemma, which has been around previously in closely
related versions, and asserts that $E$ indeed locally minimizes $\mathcal{F}_H^U$
in the described situation. The second construction works in arbitrary dimension
$n\ge2$ but draws on some more background from the theory of variational mean
curvatures and is not explicit to the same extent. Specifically, it relies on
Barozzi's formula \cite{Barozzi1994} for an $\L^1$ optimal variational mean
curvature $H$ and the minimizers-contain-balls
lemma of Tamanini \& Giacomelli \cite{Tamanini1989Approx} in order to estimate
$H$ and infer $H\in \L^p(U)$ whenever $p>n$, $\alpha>\alpha_\mathrm{opt}(n,p)$.

The results of this paper are partially contained in the second author's master thesis \cite{schuett2022master}, which has been supervised by the first author.

\section{Preliminaries}

\subsection{General notation and overall assumptions}

\textbf{Overall assumptions.} Throughout this paper, we consider a dimension $2\le n\in\mathds{N}$. Moreover, if not otherwise stated, $U\subseteq \R^n$ is an open subset of the $n$-dimensional Euclidean space.

\vspace{0.5em}

\textbf{Basic notation.} In the following, $\B_r(x)$ denotes the open ball in $\R^n$ with radius $r>0$ and center $x\in\R^n$. In case the dimension of the ball is not clear, we use $\B^k_r(x')$ for the open ball with radius $r$ and center $x'\in\R^k$ in $\R^k$ with $k\in\{1,\dots,n\}$. For $x\in\R^n$, the symbol $\cx$ denotes the $(n{-}1)$-dimensional vector $(x_1,\dots,x_{n-1})$. Moreover, $\C_r(x)$ is the open cylinder $\B_r^{n-1}(\cx)\times (x_n-r,x_n+r)$ with center $x\in\R^n$ and with height and radius $r>0$ in $\R^n$. For a constant in $(0,\infty)$ depending only on values $t_1,\dots,t_N\in \R$ with $N\in\mathds{N}$, we write $\mathrm{c}(t_1,\dots,t_N)$.

\vspace{0.5em}

\textbf{Measure-theoretic notation}. We follow standard notations and denote the $s$-dimensional Hausdorff measure with $s\in[0,\infty)$ and the $n$-dimensional Lebesgue measure on $\R^n$ by $\mathcal{H}^s$ and $\mathcal{L}^n$, respectively. The set $\mathcal{M}^n$ denotes the set of all Lebesgue-measurable subsets of $\R^n$. We abbreviate $|F|\coleq\mathcal{L}^n(F)$ for sets $F\in\mathcal{M}^n$. The notation $|\mu|$ denotes the variation measure of a Radon measure $\mu$. For $\alpha\in[0,1]$ and $F\in\mathcal{M}^n$, we denote the set of points of density $\alpha$ of $F$ by $F(\alpha)$.

\vspace{0.5em}

\textbf{Notation for functions.} For $\alpha\in(0,1]$, $N\in\mathds{N}$ and $f\in\C^{0,\alpha}(U;\R^N)$, we denote the Hölder constant of $f$ on $U$ by $C_f^\alpha$. For a set $G\subseteq\R^n$, $\mathds{1}_G$ denotes the characteristic function of $G$. In measure-theoretic contexts, we identify functions and sets which coincide $\mathcal{L}^n$-a.e. and call them representations of each other.

\subsection{The perimeter}
For a Lebesgue-measurable set $E\subseteq \R^n$, the \textbf{perimeter} of $E$ in $U$ is defined by 
\begin{align*}
    \P(E,U)\coleq \sup\left\{\int_E \mathrm{div}\varphi\dx: \varphi\in\C^1_{\mathrm{cpt}}(U;\R^n), \|\varphi\|_{\C(U)}\le 1\right\}\,.
\end{align*}
Abbreviated, we write $\P(E)$ instead of $\P(E,\R^n)$. The perimeter is in fact the total variation of the derivative of a characteristic function, more precisely, $\mathds{1}_E$ is a $\mathrm{BV}_{\loc}(U)$-function with finite derivative measure on $U$ if and only if the perimeter of $E$ in $U$ is locally finite and in this case, $|\mathrm{D}\mathds{1}_E|(U)=\P(E,U)$. Therefore, it makes sense to extend the perimeter notion by $\P(E,B)=|\mathrm{D}\mathds{1}_E|(B)$ for Borel sets $B\subseteq \R^n$ whenever there exists an open neighborhood of $B$ such that $E$ has locally finite perimeter in this set. Then the perimeter $P(E,\cdot)$ becomes a Radon measure, and it is lower semi continuous with respect to the $\L^1_\loc$-convergence in the first argument.
\\
A decisive advantage of the perimeter is the possibility to measure the area of a set independently of null set changes in a good sense. Indeed, according to the divergence theorem, each set $E\subseteq\R^n$ with Lipschitz boundary in $U$ is a set of locally finite perimeter in $U$ and $\P(E,U)=\mathcal{H}^{n-1}(\partial E\cap U)$.
\\
The following properties of perimeters are well-known and can be found in \cite[Proposition 3.38]{ambrosio2000}, for instance.
 
\begin{lemma}[Properties of the perimeter] \label{Lemma: Properties of Perimeters}
    If $E\subseteq\R^n$ is a set of locally finite perimeter in $U$, $F\in\mathcal{M}^n$ and $B,B'\subseteq U$ are Borel sets, then 
\begin{enumerate}[label=\rm\roman*)]
    \item $\P(E,B)\le \P(E,B')$ if $B\subseteq B'$ with equality if $E\ssubset B$, \label{Periemter monotony}
    \item if $|(E\sd  F)\cap U'|=0$ for some open set $U'\subseteq U$ with $B\subseteq U'$, then $F$ is a set of locally finite perimeter in $U'$ with $\P(E,B)=\P(F,B)$, in particular, $\P(E,B)=\P(\R^n\setminus E,B)$, \label{Perimeter equality}
    \item if $F$ is also a set of locally finite perimeter in $U$, then $\P(E\cap F,B)+\P(E\cup F,B)\le \P(F,B)+\P(E,B).$\label{Perimeter inequality}
\end{enumerate}
\end{lemma}

\subsection{Reduced boundaries} 

Let $E\subseteq\R^n$ be a Lebesgue-measurable set. We denote the reduced boundary of $E$, which is defined as in \cite[Definition 3.54]{ambrosio2000} and taken in the largest open set such that $E$ has locally finite perimeter in that set, by $\partial^\ast E$. By $\nu_E$, we denote the weak outward unit normal of $E$.
\\
If $x\in\partial E$ and $E$ is of class $\C^1$ near $x$, then the weak outward unit normal equals the strong one at $x$ and $x\in\partial^\ast E$. Especially, $\partial E = \partial^\ast E$ whenever $E$ has $\C^1$ boundary.
For general Lebesgue-measurable sets, De Giorgi's structure theorem says $\P(E,\cdot)=\mathcal{H}^{n-1}\mres \partial^\ast E$ on the largest open set where $E$ has locally finite perimeter. A conclusion from this statement is the generalized divergence theorem which states
\begin{align*}
 \int_E\mathrm{div}\,\Phi\dx = \int_{\partial^\ast\!E} \Phi\cdot \nu_E\,\d\mathcal{H}^{n-1}   
\end{align*}
for all $\Phi\in\mathrm{W}^{1,1}(U;\R^n)\cap \C_\mathrm{cpt}(U;\R^n)$ whenever $E$ has locally finite perimeter in $U$. Moreover, Federer's structure theorem allows to identify the reduced boundary with the measure-theoretic boundary and with the set of all points of density $\frac{1}{2}$.
\\
The reduced boundary of a set $E$ is the principal part of the boundary in the measure-theoretic sense since it is invariant under null set changes of $E$. In order to investigate the regularity, it is reasonable to choose the best possible representation $E^\ast$ of $E$ such that $\partial E^\ast\setminus \partial^\ast E$ is minimized in the sense that it holds $\partial E^\ast=\closure{\partial^\ast E}$ and for every other representation $E'$ of $E$, it holds $\partial E'\supseteq 
\closure{\partial^\ast E}$. This can be realized by defining $E^\ast$ as the measure theoretic interior $E(1)$ of $E$. Moreover, this representation satisfies $0<|E^\ast\cap\B_r(x)|<|\B_r(x)|$ for all $x\in\partial E^\ast$ and $r>0$. If not otherwise stated, we will always assume sets to have this representation.

\vspace{0.5em}

For more background on BV theory, see \cite{ambrosio2000}, \cite{giusti1984}, and \cite{maggi2012}.

\subsection{Variational mean curvatures}

Variational mean curvatures generalize the concept of mean curvatures for boundaries of arbitrary sets. In contrast to mean curvatures, the variational mean curvature is defined on a neighborhood of the surface one actually wants to describe. Initially, we consider sets with constant variational mean curvature and state a useful result before generalizing the concept. 
\\
Let $\lambda>0$ and $E\subseteq \R^n$ be a set of finite perimeter and finite volume. Consider the minimization problem 
\begin{equation}\label{Minimization problem constant VMC}
    \inf\left\{\mathcal{F}_\lambda (F): F\subseteq E,\, F\in\mathcal{M}^n\right\}\,, \tag{$\P_\lambda$}
\end{equation}
with 
\begin{equation*}
    \mathcal{F}_\lambda(F)\coleq \P(F)+\lambda|E\setminus F| \qquad \text{for } F\in\mathcal{M}^n\,,\, F\subseteq E.
\end{equation*}

The following result has been obtained in \cite[Lemma 2.4]{Tamanini1989Approx}.

\begin{lemma}[Minimizers contain balls]\label{lemma: Minim contain balls}
Let $E\subseteq\R^n$ be a set of finite perimeter and finite volume. If there exist $x\in E$, $r>0$ such that $\B_r(x)\subseteq E$, then
\begin{equation*}
    \mathcal{F}_\lambda(F\cup \B_r(x))\le \mathcal{F}_\lambda(F)
\end{equation*}
for all $\lambda\ge\frac{n}{r}$ and all $F\subseteq E$. If we have $\lambda>\frac{n}{r}$, then equality in the estimate above
\textup{(}as it occurs specifically for a minimizer $F$ of \eqref{Minimization problem constant VMC}\textup{)} implies  $\B_r(x)\subseteq F$.
\end{lemma}

For the convenience of the reader, we explicate the proof.

\begin{proof}
We first show that the following inequality holds true for all $G\subseteq \B_r(x)$:
\begin{equation}\label{EQ: minimizers conatin balls 2}
    \P( \B_r(x))\le \P(G)+\frac{n}{r}|\B_r(x)\setminus G|\,.
\end{equation}
In other words, $\B_r(x)$ is a minimizer of $\mathcal{F}^{\frac{n}{r}}\coleq\P(\,\cdot\,)-\frac{n}{r}|\cdot|$ among subsets of $\B_r(x)$. According to the isoperimetric inequality, it suffices to show that $\B_r(x)$ minimizes $\mathcal{F}^{\frac{n}{r}}$ among balls with radius in $[0,r]$ which is straightforward to verify. 
\\
Finally, for $\lambda\ge \frac{n}{r}$ and Lebesgue-measurable sets $F\subseteq E$ of finite perimeter, it follows
\begin{align*}
    \mathcal{F}_\lambda(F\cup \B_r(x)) - \mathcal{F}_\lambda(F) &=\P( F\cup \B_r(x))-\P(F)-\lambda|\B_r(x)\setminus F|\\
    &\le \P(\B_r(x))-\P(\B_r(x)\cap F)-\lambda|\B_r(x)\setminus F|\\
    &\le \left(\frac{n}{r}-\lambda\right)|\B_r(x)\setminus F|\le 0\,,
\end{align*}
where we used $\B_r(x)\subseteq E$, Lemma \ref{Lemma: Properties of Perimeters} \ref{Perimeter inequality} and \eqref{EQ: minimizers conatin balls 2} for $G=F\cap \B_r(x)$. Equality can only appear if $\B_r(x)\setminus F$ is a null set in the case $\lambda>\frac{n}{r}$. The choice of representation implies $\B_r(x)\subseteq F$.
\end{proof}

The functional $\mathcal{F}_\lambda$ is generalized by the functional $\mathcal{F}_H$, defined in \eqref{EQ: Massari functiional definiton}, for $H\in\L^1(\R^n)$. In \cite{massari1974esistenza}, Massari showed that the boundary of a suitable regular minimizer of $\mathcal{F}_H$ has mean curvature $H$ whenever $H$ is continuous. This will be explicated later in Section \ref{Sec: Dim 2 regularity} and motivates the definition of (local) variational mean curvatures given in \cite[Definition 1.1]{Barozzi1994} and \cite[p.\@ 197]{GMT1993boundaries}.

\begin{definition}[(Local) Variational mean curvatures]\label{Def: VMC}
\label{Def: Set of mean curvatures}
Let $E\subseteq\R^n$ be a set of finite perimeter and $H\in\L^1(\R^n)$. We call $H$ a \textup{(}global\textup{)} \textbf{variational mean curvature} of\/ $E$ if
\begin{align*}
    \mathcal{F}_H(E)\le \mathcal{F}_H(F)\qquad \text{for all } F\in\mathcal{M}^n\,,
\end{align*}
where
\begin{align}\label{EQ: Massari functiional definiton}
    \mathcal{F}_H(F)\coleq\P(F)-\int_F H(x) \dx\,.
\end{align}
For $p\in{[1,\infty]}$ we denote the set of all variational mean curvatures of\/ $E$ in $\L^p(E)$ by $\mathds{H}^p(E)$.
\\
Moreover, we say that a set $\tilde{E}\subseteq \R^n$ of finite perimeter in $U$ has \textbf{\textup{(}local\/\textup{)} variational mean curvature \boldmath{$\tilde{H}\in\L^1(U)$} in \boldmath{$U$}} if 
\begin{align*}
    \mathcal{F}_{\tilde{H}}^U(\tilde{E})\le \mathcal{F}_{\tilde{H}}^U(F)\qquad \text{for all }F\in\mathcal{M}^n \text{ satisfying } F\sd  E\ssubset U\,,
\end{align*}
where 
\begin{align*}
    \mathcal{F}_{\tilde{H}}^U (F)\coleq \P(F,U)-\int_{U\cap F}\tilde{H}\dx\,.
\end{align*}
We denote the set of all variational mean curvatures of\/ $\tilde{E}$ in $U$ which are contained in $\L^p(\tilde{E})$ by $\mathds{H}^p(\tilde{E},U)$.
\end{definition} 

\begin{rem}\label{Rem: local VMC}
Clearly, we have $\mathds{H}^1(E) \subseteq\mathds{H}^1(E,U)$, where we identify $H\in \mathds{H}^1(E)$ with $H|_U\in\mathds{H}^1(E,U)$. However, we warn the reader that $\mathds{H}^1(E) \subsetneq\mathds{H}^1(E,U)$ may happen even for $U=\R^n$, that is a local variational mean curvature in $\R^n$ is not necessarily a global variational mean curvature.
\end{rem}

Next we record an elementary lemma, which we have not found in the existing literature.

\begin{lemma}\label{Lemma: composed curvtaure}
Let $E\subseteq\mathbb{R}^n$ be a set of finite perimeter in $U$ and $H_1,H_2\in\mathbb{H}^1(E,U)$. Then the composed function $\widehat H\coleq H_1\mathds{1}_{E\cap U}+H_2\mathds{1}_{U\setminus E}$ is also in $\mathbb{H}^1(E,U)$. If $E$ is a set of finite perimeter with $E\Subset U$, $H_1\in\mathbb{H}^1(E,U)$, and $H_2\in \mathbb{H}^1(E)$, then $\Tilde{H}\coleq H_1\mathds{1}_{E}+H_2\mathds{1}_{\R^n \setminus E}$ is in $\mathbb{H}^1(E)$.
\end{lemma}

\begin{proof}
First we assume $E$ to have finite perimeter in $U$ and $H_1,H_2\in\mathbb{H}^1(E,U)$. Evidently, we have $\widehat H\in\mathrm{L}^1(U)$. For a set $F\in\mathcal{M}^n$ of finite perimeter with $F \triangle E\Subset U$, we compute
\begin{multline*}
   \qquad  \mathcal{F}^U_{\widehat H}(E) + \mathcal{F}^U_{H_2}(E) \\
\begin{aligned}
        =\;&\mathcal{F}^U_{H_1}(E) + \mathcal{F}^U_{H_2}(E)\\
    \le\; &\mathcal{F}^U_{H_1}(F\cap E) + \mathcal{F}^U_{H_2}(F\cup E)\\
     = \;&\P(E\cap F,U)+ \P(E\cup F,U) -\int_{E\cap F\cap U}H_1\dx - \int_{(E\cup F)\cap U} H_2\dx\\
     \le \;&\P(E,U) + \P(F,U) - \int_{F\cap U}\widehat H\dx - \int_{E\cap U} H_2\dx\\
    = \;&\mathcal{F}^U_{\widehat H}(F)+\mathcal{F}^U_{H_2}(E),
\end{aligned}
\end{multline*}
where we used Lemma \ref{Lemma: Properties of Perimeters} \ref{Perimeter inequality}.
Subtracting $\mathcal{F}^U_{H_2}(E)$ on both sides shows the first claim. Now, assume $E$ to be a of finite perimeter with $E\Subset U$ and $H_2\in\mathbb{H}^1(E)$. Then $\mathcal{F}_{H_1}^U(G)=\mathcal{F}_{\tilde H}^U(G)= \mathcal{F}_{\tilde H}(G)$ holds true for all measurable sets $G\subseteq E$. Hence, $\mathcal{F}_{\tilde H}(E) + \mathcal{F}_{H_2}(E)\le \mathcal{F}_{\tilde H}(F)+\mathcal{F}_{H_2}(E)$ can be concluded as before for all $F\in\mathcal{M}^n$.
\end{proof}

Since one can always modify a given variational mean curvature of a set $E$ by increasing its values on $E$ and decreasing its values outside $E$ without leaving the set $\mathds{H}^1$ of curvatures of $E$, one cannot hope to extract too much information on $E$ from an arbitrary variational mean curvature. However, the definition of variational mean curvatures is underpinned by the fact that for each Lebesgue-measurable set $E$ one can construct a certain optimal variational mean curvature, which may indeed yield better information. We will now present the construction of this optimal curvature from \cite{Barozzi1994}. The idea is to define a 'small' variational mean curvature for each set of finite perimeter based on the approximation of the set by subsets with constant variational mean curvature, i.e., by minimizers of \eqref{Minimization problem constant VMC}.

\begin{con}[Construction of the optimal variational mean curvature]\label{Construction H_E}
Let $E\subseteq \R^n$ be a set of finite perimeter and $h_E$ an arbitrary function in $\L^1(E)$ such that $h_E>0$ a.e. on $E$. Furthermore, let $\mu\coleq\mu_E$ be the positive and finite measure $h_E\mathcal{L}^n\mres E$. For an arbitrary number $\lambda>0$, we define the functional
\begin{align*}
    \mathcal{F}^\mu_\lambda(F) \coleq \P(F) + \lambda \mu(E\setminus F), \quad F\subseteq E, \; F\in\mathcal{M}^n
\end{align*}
and consider the minimization problem
\begin{align}\label{Construction minimization problem}
    \inf\left\{\mathcal{F}^\mu_\lambda (F): F\subseteq E,\, F\in\mathcal{M}^n\right\}.\tag{$\mathrm{CP}_\lambda$}
\end{align}
Notice that $\mathcal{F}_\lambda^\mu = \mathcal{F}_\lambda$ if $E$ is a set of finite volume and $h_E$ is chosen as $\mathds{1}_E$ such that \eqref{Construction minimization problem} turns into \eqref{Minimization problem constant VMC}. The following properties are easy to verify.
\begin{enumerate}[label=$\bullet$]
    \item There exists a \textup{(}not necessarily unique\textup{)} set $E_\lambda\subseteq E$, $E_\lambda\in\mathcal{M}^n$ which attains the minimum of \eqref{Construction minimization problem}.
    \item $E_\lambda\subseteq E_\gamma$ for all $0 < \lambda < \gamma$.
    \item $E\setminus\left(\bigcup_{\lambda\in\mathds{Q}_{>0}} E_\lambda\right)$ is a null set.
\end{enumerate} 
Finally, for a fixed choice of $(E_\lambda)_{\lambda\in\mathds{Q}_{>0}}$, we define the function
\begin{align}\label{Barozzi-formula}
    H_E(x)\coleq \inf\left\{\lambda h_E(x) : x\in E_\lambda,\, \lambda\in\mathds{Q}_{>0}\right\}
\end{align}
for $x\in E$. It is left to define $H_E$ on $\R^n\setminus E$. With regard to the fact that $H\in\mathds{H}^1(E)$ implies $-H\in\mathds{H}^1(\R^n\setminus E)$, it is reasonable to set
\begin{align*}
    H_E(x)\coleq -H_{\R^n\setminus E}(x)
\end{align*}
for $x\in \R^n\setminus E$, where the previous steps were used for $\R^n\setminus E$ instead of $E$ to construct $H_{\R^n\setminus E}$ on $\R^n\setminus E$ with corresponding a.e. positive $h_{\R^n\setminus E}\in\L^1(\R^n\setminus E)$.
\end{con}

For our purposes it is relevant that this construction allows for estimating $H_E$ via \eqref{Barozzi-formula} and Lemma \ref{lemma: Minim contain balls} on balls contained in $E$. This will enable us to construct the counterexamples for which the Hölder exponent in Massari's regularity theorem depends in a fairly sharp way on the integrability of the (optimal) curvature.

We also record the announced optimality property of $H_E$, which has been established in \cite[Theorem 2.1, Remark 2.1]{Barozzi1994}.
  
\begin{theo}[$H_E$ is an $\L^1$ optimal curvature]
  Let $E\subseteq \R^n$ be a set of finite perimeter. Then, for any choice of $h_E\in\L^1(E)$ and $(E_\lambda)_{\lambda\in\mathds{Q}_{>0}}$, the function $H_E$ from \eqref{Barozzi-formula} is a \textup{(}global\textup{)} variational mean curvature of\/ $E$ with $\P(E)=\|H_E\|_{\L^1(E)}\le\|H\|_{\L^1(E)}$ and $\P(E) = \|H_E\|_{\L^1(\R^n\setminus E)}\le \|H\|_{\L^1(\R^n\setminus E)}$ for all $H\in\mathds{H}^1(E)$.
\end{theo}

Moreover, if $E$ has finite volume and if $\mathds{H}^p(E)\neq\emptyset$ holds for $p\in {(1,\infty)}$, then \cite[Theorem 3.2]{Barozzi1994} asserts that $H_E$ is even the unique minimizer of the $\L^p(E)$-norm in $\mathds{H}^p(E)$.

\subsection{Regularity theorems}

Massari's regularity theorem was first obtained in \cite[Theorem 3.1, Theorem 3.2]{massari1975frontiere} and is now restated as follows; compare also \cite[Theorem 3.6]{massari1994variational}.

\begin{theo}[Massari's regularity theorem]\label{Theo: strong Regularity}
Let $p\in{(n,\infty]}$, $\alpha = \frac{1}{4}\big(1-\frac{n}{p}\big)$, $U\subseteq \R^n$ be an open set and $E\subseteq \R^n$ a set of finite perimeter in $U$.\footnote{For $p=\infty$, we identify $\frac{1}{4}\big(1-\frac{n}{p}\big)$ with $\frac{1}{4}$.} If 
there exists $H\in\L^p(U)$ such that $H$ is a local variational mean curvature of $E$ in $U$, that is, if $E$ minimizes the functional $\mathcal{F}_H^U(F)$
among all $F\in\mathcal{M}^n$ with $F\sd  E\ssubset U$, then the following properties are satisfied.
\begin{enumerate}[label=\rm\roman*)]
    \item $U\cap \partial^\ast E$ is an $(n-1)$-dimensional $\C^{1,\alpha}$-manifold relatively open in $U\cap \partial E$.\footnote{More precisely, for each $x\in U\cap \partial^\ast E$ there exists an open neighborhood $V\subseteq U$ of $x$ such that $V\cap \partial^\ast E=V\cap \partial E$ can be represented as a rotated and translated graph of a $\C^{1,\alpha}$-function and $E\cap V$ is the rotated and translated subgraph of this function.}
    \item For all $s\in(n-8,n]$, it holds $\mathcal{H}^s((\partial E\setminus \partial^\ast E)\cap U)=0$, where $\mathcal{H}^s\coleq\mathcal{H}^0$ for $s<0$.
\end{enumerate}
\end{theo}

\begin{rem}\label{Rem: regularity theorem of Massari}
If $U\subseteq \R^n$ is an open and bounded set and both, $U$ and $E\cap U$ are sets of finite perimeter, then $\mathds{H}^1(E\cap U)\subseteq \mathds{H}^1(E,U)$. Hence, if $H_{E\cap U}$ and $H_{U\setminus E}$ have finite $\L^p$-norm on $E\cap U$ and $U\setminus E$, respectively, for some $p\in{[1,\infty)}$, then $H_{E\cap U} \mathds{1}_{E\cap U}-H_{U\setminus E}\mathds{1}_{U\setminus E}\in\mathds{H}^1(E,U)\cap \L^p(U)$ can be verified with Lemma \ref{Lemma: composed curvtaure}.
\end{rem}

In \cite[Theorem A]{bombieri1969minimal} it is proved that the Simons
cone
\begin{align*}
    C= \left\{x\in\R^8: x_1^2+x_2^2+x_3^2+x_4^2<x_5^2+x_6^2+x_7^2+x_8^2\right\}
\end{align*}
is a set with vanishing variational mean curvature in every bounded subset of $\R^8$. Since $\partial C$ is $\C^1$ except for the origin, the condition on $s$ in Massari's regularity theorem is optimal. The requirement $p>n$ is optimal as well. In fact, in \cite[Example 2.2]{massari1994variational} and \cite[Section 2]{GMT1993boundaries}, it is shown that the theorem fails for $p<n$ and $p=n$, respectively.

In order to improve on the Hölder exponent $\alpha$ in Theorem \ref{Theo: strong Regularity} we will crucially rely on the related regularity result \cite[Proposition 1]{TamaniniBoundariesofCaccioppolisets} of Tamanini for almost-minimizers of the perimeter. This result, which in itself comes with an optimal Hölder exponent, is restated next (together with some relevant notation).

\begin{notation}
For a set $E\subseteq\R^n$ of locally finite perimeter and a bounded open set $U\subseteq\R^n$, we set
\begin{align*}
    \Xi(E,U) &\coleq \inf\left\{\P(F,U): F\in\mathcal{M}^n\,,\,F\sd  E\ssubset U\right\},\\
    \Psi(E,U)&\coleq \P(E,U)-\Xi(E,U).
\end{align*}
\end{notation}

\begin{rem}\label{Rem: Minimizer estimate pertubation}
    Clearly, $\Psi$ is monotonously increasing in the second component with respect to the subset relation. Moreover, for all sets $E\subseteq \R^n$ of locally finite perimeter and bounded open sets $U\subseteq \R^n$, there exists a set $A\subseteq\R^n$ satisfying $A\setminus U = E\setminus U$ which minimizes the perimeter in $U$ with boundary datum $E$. This follows by the direct method of calculus of variations; see, for instance, \cite[Theorem 1.20]{giusti1984}. In particular, for all $F\in\mathcal{M}^n$ with $F\sd  E\ssubset U$, it holds
    \begin{align*}
        \P(A,\closure{U}) \le \P(F,\closure{U}) =  \P(F,U) +\P(E, \partial U).
    \end{align*}
    Hence, $A$ satisfies $\P(A,\closure{U})\le \Xi(E,U)+\P(E, \partial U)$.
\end{rem}

\begin{theo}[Tamanini's regularity theorem; a-priori-$\mathrm{C}^1$ case]\label{Theo: Tamanini Characterization}
 Let $\alpha\in (0,1)$ and $E\subseteq \R^n$ be a set such that $\partial E\cap U$ is $\C^1$. Then $\partial E\cap U$ is locally of class $\C^{1,\alpha}$ if and only if for each $x\in \partial E \cap U$ there exists a neighborhood $V$ of $x$ and constants $C,R>0$ such that 
 \begin{align} \label{EQ: Tamamninis regularity theorem}
     \Psi(E,\B_r(y)) \le Cr^{n-1+2\alpha},
\end{align}
holds true for all $y\in \partial E\cap V$ and $0<r<R$.
\end{theo}

\begin{rem}
  The estimate $\Psi(E,\B_r(x))\le \mathrm{c}(p) \|H\|_{\L^p(U)}r^{n\frac{p-1}{p}}$ holds true for all $H\in \mathds{H}^p(E,U)$ and $B_r(x)\subseteq U$ according to Hölder's inequality. Hence, it is immediate that the optimal Hölder exponent in Massari's regularity theorem is greater than or equal to $\frac{1}{2}\big(1-\frac{n}{p}\big)$.
\end{rem}

\section{\boldmath Optimal \texorpdfstring{$\C^{1,\alpha}$}{Calpha} regularity}\label{Sec: Improvement HE}

In this section, we improve on the Hölder exponent in Massari's regularity theorem.

\subsection{Preparatory lemmas}

First we deal with two technical lemmas, where the second one provides us with suitable local rotations
which transform to a situation with horizontal tangent space and do not change the $\C^{1,\alpha}$ Hölder constant too much.

\begin{lemma}[Hölder continuity transferred from and to the unit normal]\label{Lemma: Hölder continuity tansmits from and to normal}
Let $\alpha\in(0,1]$, $\Omega\subseteq\R^{n-1}$ be an open set and $f\in\C^1(\closure{\Omega})$ with $\|\nabla f\|_{\C(\Omega)}<\infty$. Let $F$ denote the graph mapping of $f$. Then
$f$ is in $\C^{1,\alpha}(\Omega)$ if and only if the unit normal of the graph of $f$, that is, the vector field 
\begin{align*}
    \nu: F(\Omega)\to\R^n;\; x\mapsto \frac{(-\nabla f(\cx),1)}{\sqrt{1+|\nabla f(\cx)|^2}}\,,
\end{align*}
is in $\C^{0,\alpha}(F(\Omega);\R^n)$. Moreover, in this case we have the inequalities
\begin{align*}
    C^\alpha_{\nabla f} \le \mathrm{c}(\|\nabla f\|_{\C(\Omega)},\alpha)C^\alpha_\nu \qquad\text{and}\qquad C^\alpha_\nu \le  C^\alpha_{\nabla f}\,.
\end{align*}
\end{lemma}

\begin{proof}
  We start with the forward implication. Since $V:\R^{n-1}\to\R^n;\;x'\mapsto \frac{(-x',1)}{\sqrt{1+|x'|^2}}$ is Lipschitz continuous with Lipschitz constant $1$, it follows
\begin{align*}
    |\nu(x)-\nu(y)| = |V(\nabla f(\cx))-V(\nabla f(\closure{y}))| \le |\nabla f(\cx)-\nabla f(\closure{y})| \le C^\alpha_{\nabla f} |x-y|^\alpha
\end{align*}
for all $x,y\in F(\Omega)$.
\smallskip

Now we turn to the backward implication. Let $\cx,\closure{y}\in \Omega$ and $x\coleq(\cx,f(\cx))$, $y\coleq(\closure{y},f(\closure{y}))\in F(\Omega)$. We first notice $\nabla f(\closure{z})=-\frac{\closure{\nu}(z)}{\nu_n(z)}$, $|\closure{\nu}(z)|\le 1$ and that $\frac{1}{|\nu_n(z)|}=\sqrt{1+|\nabla f(\closure{z})|^2}$ is bounded from above by $\sqrt{1+\|\nabla f\|_{\C(\Omega)}^2}$ for each $z\in F(\Omega)$. Thus, by estimating
\begin{align*}
    |\nabla f(\cx)-\nabla f(\closure{y})| &= \left|\frac{\closure{\nu}(x)}{\nu_n(x)}-\frac{\closure{\nu}(y)}{\nu_n(y)}\right|\\
    &\le \frac{1}{|\nu_n(x)|}\left|\closure{\nu}(x)-\closure{\nu}(y)\right|+\frac{|\closure{\nu}(y)|}{|\nu_n(x)\nu_n(y)|}\left|\nu_n(y)-\nu_n(x)\right|\\
    &\le 2(1+\|\nabla f\|^2_{\C(\Omega)}) \left|\nu(x)-\nu(y)\right|\\
    &\le 2(1+\|\nabla f\|^2_{\C(\Omega)})C^\alpha_\nu|x-y|^\alpha\\
    &= 2(1+\|\nabla f\|^2_{\C(\Omega)})C^\alpha_\nu\left(|\cx-\closure{y}|^2+|f(\cx)-f(\closure{y})|^2\right)^{\frac{\alpha}{2}}\\
    &\le 2(1+\|\nabla f\|^2_{\C(\Omega)})^{1+\frac{\alpha}{2}}C^\alpha_\nu|\cx-\closure{y}|^\alpha,
\end{align*}
we arrive at the claim.
\end{proof}

\begin{lemma}[Existence of good graph representations of $\C^{1,\alpha}$-sets]\label{Lemma: Existence of good graph representation of Hölder cont functions}
Let $\alpha\in(0,1]$, $f\in \C^{1,\alpha}(\Omega)$ with $\|\nabla f\|_{\C(\Omega)}<\infty$ for an open set $\Omega\subseteq \R^{n-1}$, and let $F$ denote the graph mapping of $f$. 
\\
For all $x_0\in F(\Omega)$, there exists a constant $R>0$ such that for all $x\in \C_R(x_0)\cap F(\Omega)$, there exists a rotation $T$ with $Tx=x$ such that $\C_R(x)\cap TF(\Omega)$ is the graph of a $\C^{1,\alpha}(\B_R^{n-1}(\cx))$-function with vanishing gradient at $\cx$ and Hölder constant in $\big[0, \mathrm{c}(\alpha) C^\alpha_{\nabla f}\big]$.
\end{lemma}

\begin{proof}
  Let $x_0\in F(\Omega)$. We can choose $R>0$ small enough such that $C^\alpha_{\nabla f}R^\alpha<\varepsilon$ for some $\varepsilon\in\left(0,1\right)$ and $\B^{n-1}_R(\cx)\ssubset \Omega$ for all $x\in \C_R(x_0)$. Then, according to the Hölder continuity of $\nabla f$,
\begin{align*}
    \sup_{y'\in \B^{n-1}_R(z')\cap\Omega}|\nabla f(z')-\nabla f(y')|\le C_{\nabla f}^\alpha R^\alpha<\varepsilon
\end{align*}
holds true for all $z'\in\Omega$. Moreover, since the unit normal $\nu_{F(\Omega)}$ is locally uniformly continuous on $F(\Omega)$, we can make $R$ small enough such that 
\begin{align}\label{EQ: Lemma rotation}
    |\nu_{F(\Omega)}(y)-\nu_{F(\Omega)}(z)|<\frac{1}{2}
\end{align}
for all $y\in \C_R(x_0)\cap F(\Omega)$, $z\in F(\Omega)$ with $|z-y|\le R$.
\\
In order to obtain a rotation which preserves the graph structure of $F(\Omega)$, we show for $x\in \C_R(x_0)\cap F(\Omega)$ that the orthogonal projection of $F(\Omega)\cap \B_R(x)$ on the tangent space $\mathrm{T}_xF(\Omega)$ is one-to-one. Indeed, let us consider $y,w\in  F(\Omega)\cap \B_R(x)$ such that $y-w$ is parallel to $\nu_{F(\Omega)}(x)$. According to the mean value theorem, there exists $z'\in \B^{n-1}_R(\cx)$ such that $f(\closure{y})-f(\closure{w})=\nabla f(z')\cdot (\closure{y}-\closure{w})$. Thus, the estimate
\begin{align*}
    |\closure{y}-\closure{w}| &\le |y-w| \\
    &= |\nu_{F(\Omega)}(x)\cdot(y-w)|\\
    &= \left|\frac{-\nabla f(\cx)}{\sqrt{1+|\nabla f(\cx)|^2}}\cdot (\closure{y}-\closure{w}) + \frac{f(\closure{y})-f(\closure{w})}{\sqrt{1+|\nabla f(\cx)|^2}}\right|\\
    &\le \frac{|\nabla f(\cx)-\nabla f(z')|}{\sqrt{1+|\nabla f(\cx)|^2}} |\closure{y}-\closure{w}|\\
    &\le \varepsilon |\closure{y}-\closure{w}|
\end{align*}
enforces $y=w$ according to the choice of $\varepsilon$. Hence, the projection is one-to-one and there exist a rotation $T$ with $Tx=x$ such that $\nu_{TF(\Omega)}(x)=e_n$ and $TF(\Omega)\cap \B_R(x)=\tilde{F}(\tilde{\Omega})$ for the graph mapping $\tilde{F}$ of a $\C^1(\tilde{\Omega})$-function $\tilde{f}$ and some open set $\tilde{\Omega}\subseteq \B^{n-1}_R(\cx)$ with $\cx\in \tilde{\Omega}$.
\\
    In the next step, we show that the rotation preserves Hölder continuity with control on the Hölder constant. According to Lemma \ref{Lemma: Hölder continuity tansmits from and to normal}, the outward unit normal $\nu_{F(\Omega)}$ is $\alpha$-Hölder continuous on $F(\Omega)$ with Hölder constant in $\big[0,C^\alpha_{\nabla f}\big]$ on $F(\Omega)$. Since rotations are isometric, $\nu_{TF(\Omega)}(y)=T\left[\nu_{F(\Omega)} (T^{-1}y)+x\right]-x$ is still $\alpha$-Hölder continuous on $ TF(\Omega)$ with the same Hölder constant. With \eqref{EQ: Lemma rotation} and again the isometry property, we can estimate
    \begin{align*}
    \left|1-\frac{1}{\sqrt{1+|\nabla \tilde{f}(\closure{y})|^2}}\right| &= \left|\left(\nu_{TF(\Omega)}(x)\right)_n-\left(\nu_{TF(\Omega)}(y)\right)_n\right|\\
    &\le|\nu_{TF(\Omega)}(x)-\nu_{TF(\Omega)}(y)|
    \le \frac{1}{2}
\end{align*}   
    for all $y\in \B_R(x)\cap TF(\Omega)= \tilde{F}(\tilde{\Omega})$. Hence, $|\nabla \tilde{f}|$ is bounded on $\tilde{\Omega}$ by $\sqrt{3}$. Applying Lemma \ref{Lemma: Hölder continuity tansmits from and to normal} once more, we infer that $\tilde{f}$ is a $\C^{1,\alpha}(\tilde{\Omega})$-function with Hölder constant in $\big[0, \mathrm{c}(\alpha) C^\alpha_{\nabla f}\big]$.
\\
    Finally, the estimate
\begin{equation*}
    |x_n-\tilde{f}(y')|=|\tilde{f}(\cx)-\tilde{f}(y')|
    \le \sup_{\tilde{\Omega}}|\nabla \tilde{f}| r
    = \sup_{\tilde{\Omega}}|\nabla \tilde{f}-\nabla \tilde{f}(\cx)| r
    \le \mathrm{c}(\alpha) C^\alpha_{\nabla f} r^{1+\alpha}
\end{equation*}
for all $y'\in \B_r^{n-1}(\cx)\cap \tilde{\Omega}$ allows to choose a small $\tilde{R}\in(0,R)$ independently of the choice of $x$ such that $\C_{\tilde{R}}(x)\ssubset \B_R(x)$ and $\tilde{f}(y')\in (x_n-\tilde{R},x_n+\tilde{R})$ for all $y'\in \B^{n-1}_{\tilde{R}}(\cx)\cap \tilde{\Omega}$. With the continuity of $f$, it follows $\tilde{\Omega}\supseteq \B^{n-1}_{\tilde{R}}(\cx)$. Thus, $TF(\Omega)\cap \C_{\tilde{R}}(x)$ is the graph of a $\C^{1,\alpha}$-Hölder continuous function on $\B^{n-1}_{\tilde{R}}(\cx)$ with Hölder constant in $\big[0, \mathrm{c}(\alpha) C^\alpha_{\nabla f}\big]$. \qedhere
\end{proof}

\subsection{Regularity up to the optimal exponent}

With Lemma \ref{Lemma: Existence of good graph representation of Hölder cont functions} at hand, we now state and prove the main result of this section.

\begin{theo}[Massari's regularity theorem with sharp Hölder exponent]\label{Theo: Improvement reg theo}
The statement of Theorem \ref{Theo: strong Regularity} holds true for all Hölder exponents $\alpha<\frac{p-n}{p+1}$.\footnote{For $p=\infty$, we identify $\frac{p-n}{p+1}$ with $1$.}
\end{theo}

\begin{proof}
Let $E,U,p,H$ be as in Theorem \ref{Theo: strong Regularity}, and set $\alpha_0\coleq\frac{1}{4}\big(1-\frac{n}{p}\big)$. In particular, we can assume that $E$ is represented by $E(1)$ such that $\partial E = \closure{\partial^\ast E}$. 
In order to apply Theorem \ref{Theo: Tamanini Characterization}, our aim is to estimate 
\begin{align}\label{EQ: Goal for improv}
   \Psi(E,\B_r(y)) \le Cr^{n-1+2\alpha_1}
\end{align}
for suitable $y\in U$, local constants $C,R>0$, $r\in(0,R)$ and for $\alpha_0<\alpha_1\in(0,1)$.

\medskip

\emph{Step 1. General framework.}
For $x\in\partial^\ast E$, by Theorem \ref{Theo: strong Regularity} there exists an open neighborhood $V$ of $x$ in $U$ such that $\partial^\ast E\cap V=\partial E\cap V$ can be represented by a rotated graph of a $\C^{1,\alpha_0}$-function and the set $E\cap V$ is the rotated subgraph of this function.
Hence, the rotated subgraph is a set with variational mean curvature $H$ in $V$ and with $\C^{1,\alpha_0}$-boundary in $V$. Since we can formulate the following proof for $V$ instead of $U$, we can w.l.o.g. assume $\partial E=\partial^\ast E$ in $U$.

\medskip

\emph{Step 2. Reduction to horizontal tangent spaces and basic $\C^{1,\alpha_0}$ estimate.} Now consider a point $x\in\partial E\cap U$ on the surface. By rotation invariance of the perimeter and the variational mean curvature, we can assume that $f:\Omega\to\R$ is the $\C^{1,\alpha_0}$-function which represents $E$ near $x$ as a subgraph for some open neighborhood $\Omega\subseteq\R^{n-1}$ of $\cx$.
 According to Lemma \ref{Lemma: Existence of good graph representation of Hölder cont functions}, we can make $R>0$ small enough such that for all $y\in \partial E\cap \C_R(x)$, we can find a rotation $T$ with $Ty=y$ such that $TE\cap \C_R(y)$ is still the subgraph of a $\C^{1,\alpha_0}(\B_R^{n-1}(\closure{y}))$-function with vanishing gradient at $\closure{y}$ and with uniformly bounded Hölder constant of the gradient, i.e., it depends on the choice of $x$ but not on the choice of $y$. Since the perimeter is invariant under translation and rotation, we can assume w.l.o.g. $\nabla f(\closure{y})=0$ and $y_n=f(\closure{y})=0$ for fixed $y\in\partial E\cap \C_R(x)$.\\
Let now $r\in (0,R)$. We show 
\begin{align}\label{EQ: Cylinder estimate}
|z_n|\le C_{\nabla f}^{\alpha_0} r^{1+\alpha_0}\coleq c(r) \qquad \text{for all } z\in \left(E\sd  \R^n_-\right)\cap \C_r(y),
\end{align}
where $\R^n_-$ denotes the lower half-space $\R^{n-1}\times{(-\infty,0)}$. Indeed, since $E$ is the subgraph of $f$ in $\C_r(y)$, it holds $|z_n|\le |f(\closure{z})|$ for all $z\in (E\sd  \R^n_-)\cap \C_r(y)$. Since we assumed $f(\closure{y})=0$ and $\nabla f(\closure{y})=0$, it follows
\begin{align*}
    |z_n| \le |f(\closure{z})-f(\closure{y})|
    \le\sup_{w'\in \B^{n-1}_r(\closure{y})}|\nabla f(w')-\nabla f(\closure{y})| r
    \le C_{\nabla f}^{\alpha_0} r^{1+\alpha_0}.
\end{align*}

\medskip

\emph{Step 3. An analogous estimate for a perimeter-minimizing competitor.}
Now, let $A$ be the perimeter minimizer in $\C_r(y)$ with boundary datum $E$. We provide a cut-off argument to ensure the $\C^{1,\alpha_0}$ estimate \eqref{EQ: Cylinder estimate} for $A$ instead of $E$. Since $A$ has boundary datum $E$ in $\C_r(y)$, we can only modify $A$ away from $\partial \C_r(y)$ in order to preserve the perimeter minimizer property with boundary datum for the cut-off.
\\
We can make $R$ smaller to ensure that $E$ has still subgraph representation in $\C_{R+\lambda}(\tilde{y})$ for all $\tilde{y}\in \C_R(x)$ and some $\lambda>0$. Fix $\varepsilon>0$. According to \eqref{EQ: Cylinder estimate} and the continuity of $\partial E$, there exists $\delta\coleq\delta_\varepsilon\in(0,\min\{\lambda,\varepsilon\})$ such that $z\in E\cap \C_{r+\delta}(y)$ implies $z_n\le c(r)+\varepsilon$.
\begin{figure}[ht]
\centering
     \begin{tikzpicture}[scale=0.3]
     \draw[white] (5,-3)--(5,0);
\filldraw[cyan, fill=cyan!20] (5,4.98)--(10,7.77)--(10,0);
\filldraw[cyan!20, fill=cyan!20] (0,0)--(0,3.1)--(5,5);
\filldraw[cyan!20, fill=cyan!20] (0,0)--(5,5)--(10,0);

\draw (0,0)--(0,10)--(10,10)--(10,0)--(0,0);
\draw[thick, red!50] (0,0)--(10,0);
\draw[gray, dashed] (0,2)--(10,2);
\draw[gray, dashed] (0,8)--(10,8);
\node[gray] at (-1,8) {\small$c(r)$};
\node[gray] at (-1.4,2) {\small$-c(r)$};
\node at (-0.45,10) {\small$r$};
\node at (-0.8,0) {\small$-r$};
\node[red!70] at (5,-1) {\small$\B_r^{n-1}(\closure{y})$};
\node[cyan] at (5,3.5) {\small$E\cap \C_r(y)$};
\node at (5,11) {\small$\C_r(y)$};

\filldraw[thick,cyan, fill=white, variable=\x,domain=5:10,samples=20] 
  plot ({\x},
 {0.25*(\x-5)^1.5+5})
 node[right] {$f$};
\filldraw[thick,cyan, fill= cyan!20,variable=\x,domain=0:5,samples=40] 
  plot ({\x},
 {(-0.2*(-\x+5)^1.5)+0.2*cos(100*\x)+5-0.2*cos(100*5)}
 );
 \filldraw[color=white] (5,4.99)--(5,5.1)--(10,7.799)--(5,4.99);
 \filldraw[thick,cyan, fill=white, variable=\x,domain=9.8:10,samples=5] 
  plot ({\x},
 {0.25*(\x-5)^1.5+5})
 node[right] {$f$};
\filldraw (5,5) circle (0.05);
\node at (5.2,5.5) {\small$y$};
\end{tikzpicture}
\quad
     \begin{tikzpicture}[scale=0.3]

\filldraw[cyan!20, fill=cyan!20] (5,5.45)--(11,8.25)--(11,-1)--(5,-1);
\filldraw[cyan!20, fill=cyan!20] (-0,-1)--(-0,3.11)--(5,5.45)--(5,-1);
\filldraw[cyan!20, fill=cyan!20] (-1,-1)--(0,-1)--(0,3.11)--(-1,2.153);
\filldraw[cyan!20, fill=cyan!20] (11,-1)--(10,-1)--(10,7.77)--(11,8.65);

\draw (-1,-1)--(-1,11)--(11,11)--(11,-1)--(-1,-1);
\draw[thick, red!50] (-1,-1)--(11,-1);
\node at (-2.9,11) {\small$r+\delta$};
\node at (-3.4,-1) {\small$-(r+\delta)$};
\node[red!70] at (5,-2) {\small$\B_{r+\delta}^{n-1}(\closure{y})$};
\node[cyan] at (5,3.5) {\small$A_\varepsilon$};
\node at (5,12) {\small$\C_{r+\delta}(y)$};

\filldraw[thick,cyan, fill=white, variable=\x,domain=10:11,samples=20] 
  plot ({\x},
 {0.25*(\x-5)^1.5+5})
 node[right] {$f$};
\filldraw[thick,cyan, fill= cyan!20,variable=\x,domain=-1:0,samples=40] 
  plot ({\x},
 {(-0.2*(-\x+5)^1.5)+0.2*cos(100*\x)+5-0.2*cos(100*5)}
 );
 \draw[thick, cyan] (0,3.11)--(10,7.8);
\draw[gray, dashed] (-1,2)--(11,2);
\draw[gray, dashed] (-1,8)--(11,8);
 \draw (0,0)--(0,10)--(10,10)--(10,0)--(0,0);
\filldraw (5,5) circle (0.05);
\node at (5.5,5) {\small$y$};

\draw[orange, dashed] (-1,8.7)--(11,8.7);
\node[orange] at (-3.75,8.7) {\small$c(r)+\varepsilon$};

\end{tikzpicture}
\caption{Configuration in the proof of Theorem \ref{Theo: Improvement reg theo}}
    \label{Figure to improved HE}
\end{figure}
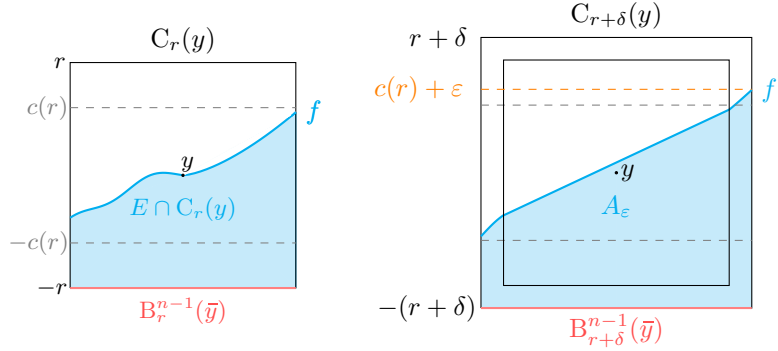

 Define the half-space $H^\gamma\coleq \{w\in \R^n: w_n\le c(r)+\gamma\}$ for $\gamma \ge 0$. For $F\in\mathcal{M}^n$ with $F \setminus \C_r(y) = E \setminus \C_r(y)$, we set $F_\varepsilon\coleq F\cap \C_{r+\delta}(y)$ and record specifically $F=F_\varepsilon$ in $\C_{r+\delta}(y)$, $E_\varepsilon=F_\varepsilon$ outside $\C_r(y)$. Moreover, $E_\varepsilon$ and $A_\varepsilon$ are sets of finite perimeter and finite volume according to Lemma \ref{Lemma: Properties of Perimeters} \ref{Perimeter inequality} since $A$ and $E$ both have finite perimeter in $U$ and $\C_{r+\delta}(y)\ssubset U$ is a set of finite perimeter. Intersecting a set of finite perimeter and finite volume with the half-space $H^\varepsilon$ reduces the perimeter. Therefore it follows
\begin{align*}
  \P(A_\varepsilon\cap H^\varepsilon,\closure{\C}_r(y))
  &= \P(A_\varepsilon\cap H^\varepsilon) - \P(E_\varepsilon, \R^n\setminus \closure{\C}_r(y))\\
  &\le \P(A_\varepsilon)  - \P(E_\varepsilon, \R^n\setminus \closure{\C}_r(y))\\
  &= \P(A,\closure{\C}_r(y))\\
  &\le \P(F,\closure{\C}_r(y))\,.
\end{align*}

Thus, the set $\tilde{A}_\varepsilon \coleq \left(A\cap H^\varepsilon\cap \C_{r+\delta}(y)\right)\cup \left(E\setminus \C_{r+\delta}(y)\right)$, which can also be written as $\tilde{A}_\varepsilon = \left(A\cap H^\varepsilon\cap \C_r(y)\right)\cup \left(E\setminus \C_r(y)\right)$, minimizes the perimeter in $\C_r(y)$ with boundary datum $E$. \\
    Since $\tilde{A}_\varepsilon\to \tilde{A}_0\coleq(A\cap H^0\cap\C_r(y))\cup(E\setminus\C_r(y))$ in $\L^1(\R^n)$ for $\varepsilon\searrow 0$, the semi-continuity of the perimeter and the minimality of $\tilde{A}_\varepsilon$ for all $\varepsilon>0$ imply that $\tilde{A}_0$ is a perimeter minimizer in $\C_r(y)$ with boundary datum $E$, too. Hence, we can replace $A$ with $\tilde{A}_0$ to ensure that $z\in A\cap \C_r(y)$ implies $z_n\le c(r)$.
    \\
    Since the perimeter of a set is equal to the perimeter of the complement, $\R^n\setminus A$ is a perimeter minimizer in $\C_r(y)$ with boundary datum $\R^n\setminus E$. Thus, by an analogous cut-off argument we may assume that $z\in (\R^n\setminus A)\cap \C_r(y)$ implies $z_n\ge-c(r)$. All in all we conclude that $z\in (A\sd \R^n_-)\cap \C_r(y)$ implies $|z_n|\le c(r)$.

\medskip

\emph{Step 4. Improved control on the deviation from minimality and $\C^{1,\alpha_1}$ regularity.}
The inner and outer trace of $E$ with respect to the cylinder $\C_r(y)$ coincide since $E$ is the subgraph of a $\C^1$ function in $\C_{R+\lambda}(y)$. In particular, Remark \ref{Rem: Minimizer estimate pertubation} implies $\P(A,\C_r(y))\le\Xi(E,\C_r(y))$ and thus $\Psi(E,\C_r(y))\le \P(E,\C_r(y))-\P(A,\C_r(y))$.
Finally, we can formulate the main argument. According to the estimate \eqref{EQ: Cylinder estimate} for $E$ and the corresponding estimate for $A$, each point $z\in (E\sd  A)\cap \C_r(y)$ satisfies $|z_n|\le c(r)$. Hence, Hölder's inequality together with the monotonicity of $\Psi$ in the second variable implies
\begin{align*}
    \Psi(E,\B_r(y))&\le \P(E,\C_r(y))-\P(A,\C_r(y))\\
    &= \mathcal{F}_H^{\C_r(y)}(E)-\mathcal{F}_H^{\C_r(y)}(A) + \int_{E\cap \C_r(y)}H\,\d z - \int_{A\cap \C_r(y)}H\,\d z\\
    &\le \int_{(E\sd  A)\cap \C_r(y)}|H|\d z\\
    &\le \int_{\B_r^{n-1}(\closure{y})} \int_{-c(r)}^{c(r)} |H(\closure{z},z_n)|\,\d z_n\,\d \closure{z}\\
    &\le \left|\B_r^{n-1}(\closure{y})\times(-c(r),c(r))\right|^{1-\frac{1}{p}} \|H\|_{\L^p(U)}\\
    &\le \mathrm{c}(n,p,\|H\|_{\L^p(U)},C_{\nabla f}^{\alpha_0})r^{(n+\alpha_0)(1-\frac1p)}\\
    &=\mathrm{c}(n,p,\|H\|_{\L^p(U)},C_{\nabla f}^{\alpha_0})r^{n-1+\alpha_0\left(1-\frac{1}{p}\right)+\frac{p-n}{p}}\,.
\end{align*}
From Theorem \ref{Theo: Tamanini Characterization} we then infer that $\partial E$ is $\C^{1,\alpha_1}$ in $U$ for $\alpha_0<\alpha_1\coleq g(\alpha_0)<1$, where $g(s)\coleq \frac{\left(1-\frac{1}{p}\right)}{2}s+\frac{p-n}{2p}$ for $s\in\R$.

\medskip

\emph{Step 5. Iteration and conclusion.}
Since $0<\frac{\left(1-\frac{1}{p}\right)}{2}<1$, the Banach fixed-point theorem implies that the sequence $(\alpha_k)_{k\in\mathds{N}}$, defined by $\alpha_k\coleq g(\alpha_{k-1})$, converges from below to the unique fixed point $\alpha_\ast$ of $g$. A rearrangement of the formula $\alpha_\ast = g(\alpha_\ast)$ shows $\alpha_\ast = \frac{p-n}{p+1}$. By repeating the argument above, we iteratively infer that $U\cap\partial E$ is $\C^{1,\alpha_k}$ for all $k\in\mathds{N}$. Hence, $U\cap\partial E$ is a $\C^{1,\alpha}$-manifold for all $\alpha<\alpha_\ast$. \qedhere
\end{proof}

\subsection[Remarks on the case of \texorpdfstring{$\L^\infty$}{Linf} curvature]{\boldmath Remarks on the case of \texorpdfstring{$\L^\infty$}{Linf} Curvature} \label{Sec: Dim 2 regularity}

Next we briefly investigate the first variation of Massari's functional at
a minimizer. In view of Theorem \ref{Theo: strong Regularity}, in case $n\le7$
or when considering the regular sets only, we can assume $\C^1$ regularity of the
minimizer and can locally represent it as a subgraph
\begin{equation}\label{EQ: DEF Ef}
  E^f\coleq\{x\in\Omega\times\R\,:\,x_n<f(\cx)\}    
\end{equation}
of $f\in\C^1(\overline\Omega)$ on a bounded open set $\Omega\subseteq\R^{n-1}$.
Then, on the cylinder
\begin{equation}\label{EQ:BD-CYL}
  C\coleq\Omega\times{({-}r,r)}
  \qquad\text{with }r>\|f\|_{\C(\Omega)}\,,
\end{equation}
Massari's functional takes the form already mentioned in the introduction
\begin{equation}\label{EQ: MF REw}
  \mathcal{F}_H^C(E^f)
  =\int_\Omega\sqrt{1+|\nabla f|^2}\,\d y-\int_\Omega\int_{-r}^{f(y)}H(y,s)\,\d s\,\d y
\end{equation}
with $H\in\L^1(C)$. If $H$ is continuous on $C$, it is standard to differentiate
this functional in the sense of the first variation and deduce, for a minimizer
$E^f$ with $f\in\C^1(\overline\Omega)$, the prescribed mean curvature equation
\begin{equation}\label{EQ:PMC}
  {-}\mathrm{div}\,\frac{\nabla f}{\sqrt{1+|\nabla f|^2}}=H(\,\cdot\,,f)
  \qquad\text{on }\Omega
\end{equation}
(with the divergence taken in the weak sense). In particular, if we even have
$f\in\C^2(\Omega)$, then $\frac1{n-1}H$ restricted to
$F(\Omega)=\{x\in\Omega\times\R\,:\,x_n=f(\cx)\}$ is the mean curvature of the
graph $F(\Omega)$ (with respect to the outward normal of $E^f$).

\medskip

If we turn to discontinuous $H\in\L^1(C)$, differentiability of \eqref{EQ: MF REw} breaks down.
However, we now record, as a minor observation, that in specific cases with
$s$-uniform integrability of $H(\,\cdot\,,s)$ at least a differential inequality
closely related to \eqref{EQ:PMC} remains valid. The precise statement is as
follows.
  
\begin{prop}[Differential inequality in case of uniform mean curvature]
    \label{PROP:DIFF-INEQ}
  Consider a bounded open set  $\Omega\subseteq\R^{n-1}$ and
  $f\in\C^1(\overline\Omega)$. If\/ $H\in\L^1(C)$ is a variational mean
  curvature of $E^f$ in $C$ with $E^f$ and $C$ from \eqref{EQ: DEF Ef} and \eqref{EQ:BD-CYL} respectively and if we have
  \begin{equation}\label{EQ:S-UNIF}
    |H(y,s)|\le\Phi(y)\text{ for }\mathcal{L}^n\text{-a.\@e.\@ }(y,s)\in C
    \qquad\text{with }\Phi\in\L^q(\Omega)\,,\,q\in{(1,\infty]}\,,
  \end{equation}
  then $\mathrm{div}\,\frac{\nabla f}{\sqrt{1+|\nabla f|^2}}$ exists weakly in
  $\L^q(\Omega)$ with
  \[
    \bigg\|\mathrm{div}\,\frac{\nabla f}{\sqrt{1+|\nabla f|^2}}\bigg\|_{\L^q(\Omega)}\le\|\Phi\|_{\L^q(\Omega)}\,.
  \]
\end{prop}

We emphasize that the main case of interest in Proposition \ref{PROP:DIFF-INEQ}
are bounded curvatures $H\in\L^\infty(C)$, for which \eqref{EQ:S-UNIF} holds and
the proposition applies with $q=\infty$ and
$\|\Phi\|_{\L^\infty(\Omega)}=\|H\|_{\L^\infty(C)}$. In contrast, we cannot
generally expect to have \eqref{EQ:S-UNIF} at hand in case of $H\in\L^q(C)$ with
$q<\infty$.

\begin{proof}[Proof of Proposition \ref{PROP:DIFF-INEQ}]
  Since $H$ is a variational mean curvature of $E^f$, the function $G(t,\varphi)\coleq\mathcal{F}_H^{C}\left(E^{f+t\varphi}\right)$, defined for $\varphi\in\C_\mathrm{cpt}^\infty(\Omega)$ and $t\in{(-\delta_\varphi,\delta_\varphi)}$, where $\delta_\varphi>0$ is small enough such that $\|f+t\varphi\|_{\C(\Omega)}<r$ for all $t\in(-\delta_\varphi,\delta_\varphi)$, attains its minimum at $t=0$. Taking into account equation \eqref{EQ: MF REw}, we can estimate
    \begin{align*}
        0&\le \limsup_{t\searrow0} \frac{G(t,\varphi)-G(0,\varphi)}{t} \\
        & = \frac{\d }{\d t}\Bigl|_{t=0}\int_{\Omega} \sqrt{1+|\nabla f+t\nabla\varphi|^2}\d y - \liminf_{t\searrow 0} \frac{1}{t} \int_\Omega \int_{f(y)}^{f(y)+t\varphi(y)}H(y,s)  \d s\d y \\   
        &\le \int_\Omega \frac{\nabla f \cdot\nabla \varphi}{\sqrt{1+|\nabla f|^2}}\d y + \|\varphi\|_{\L^{q'}(\Omega)}\|\Phi\|_{\L^q(\Omega)} 
    \end{align*}
    for $q'\coleq \frac{q}{q-1}\in[1,\infty)$ such that $q$ and $q'$ are conjugate exponents. Analogously, we get
    \begin{align*}
        0 \le \limsup_{t\searrow0} \frac{G(-t,\varphi)-G(0,\varphi)}{t} \le   - \int_\Omega \frac{\nabla f \cdot \nabla\varphi}{\sqrt{1+|\nabla f|^2}}\d y+\|\varphi\|_{\L^{q'}(\Omega)}\|\Phi\|_{\L^q(\Omega)}
    \end{align*}
for all $\varphi\in\C^\infty_\mathrm{cpt}(\Omega)$. In conclusion we derived the estimate
    \begin{align*}
        \left|\int_\Omega \frac{\nabla f \cdot \nabla \varphi}{\sqrt{1+|\nabla f|^2}}\d y\right|\le\|\Phi\|_{\L^q(\Omega)}  \|\varphi\|_{\L^{q'}(\Omega)}
    \end{align*}
    for all $\varphi\in\C_\mathrm{cpt}^\infty(\Omega)$. Thus, 
\[
\C^\infty_\mathrm{cpt}(\Omega)\to\R;\, \varphi\mapsto  \int_\Omega \frac{\nabla f\cdot \nabla \varphi}{\sqrt{1+|\nabla f|^2}}\d y
\]
can be extended to a bounded linear functional on $\L^{q'}(\Omega)$. By $\L^p$ duality, there exists $g\in\L^q(\Omega)$ such that $\int_\Omega \frac{\nabla f \cdot \nabla \varphi}{\sqrt{1+|\nabla f|^2}}\d y= \int_\Omega g \varphi\d y$ for all $\varphi\in\C^\infty_\mathrm{cpt}(\Omega)$ and $\|g\|_{\L^q(\Omega)}\le \|\Phi\|_{\L^q(\Omega)} $. By definition, $\mathrm{div}\frac{\nabla f}{\sqrt{1+|\nabla f|^2}}$ exists weakly in $\L^q(\Omega)$ and coincides with $g$.
\end{proof}

\begin{rem}
  In the setting of Proposition \ref{PROP:DIFF-INEQ} with $q=1$, a similar reasoning gives existence of $\mathrm{div}\,\frac{\nabla f}{\sqrt{1+|\nabla f|^2}}$ as a finite Radon measure on $\Omega$ with its total variation bounded by $\|\Phi\|_{\L^1(\Omega)}$. However, we do not pursue this case any further.
\end{rem}

In view of the last result, we are able to improve on Theorem \ref{Theo: Improvement reg theo} and obtain even $\C^{1,1}$ regularity in the special case $H\in\L^\infty(U)$, $n=2$, where the divergence is simply the derivative.

\begin{prop}[$\C^{1,1}$ regularity for the case of $\L^\infty$ curvature in $\R^2$]\label{Prop: n=2 special regularity}
Let $E\subseteq\R^2$ be a set of finite perimeter in an open set $U\subseteq \R^2$ with variational mean curvature $H\in \L^\infty(U)$ in $U$. Then $\partial^\ast\!E=\partial E$ is a $\C^{1,1}$-manifold in $U$.
\end{prop}

\begin{proof}
By Massari's regularity theorem, $\partial^\ast E\cap U$ is of class $\C^1$ with $\partial^\ast\!E=\partial E$ in $U$. We localize and exploit that the perimeter and the variational mean curvature are invariant under translation and rotation. Hence, it suffices to consider $\partial E$ in a rectangle $C=\Omega\times{(-r,r)}\Subset U$ over a bounded open interval $\Omega\subseteq\R$ such that $E$ coincides inside $C$ with the subgraph of some function $f\in\C^1(\closure{\Omega})$ with $r>\|f\|_{\C(\Omega)}$.
\\
By Proposition \ref{PROP:DIFF-INEQ}, we have $\frac{f'}{\sqrt{1+f'^2}}\in\mathrm{W}^{1,\infty}(\Omega)$, which means that $\frac{f'}{\sqrt{1+f'^2}}$ is a Lipschitz function. Since $h(s)\coleq\frac{s}{\sqrt{1-s^2}}$ is Lipschitz continuous on $\left[0,\frac{M}{\sqrt{1+M^2}}\right]$, where $M\coleq \|f'\|_{\C(\Omega)}<\infty$, the function $f'=h\circ \frac{f'}{\sqrt{1+f'^2}}$ is also Lipschitz.
\end{proof}

However, the last example of \cite[Remark 3.6]{massari1994variational} shows that the existence of
an $\L^\infty$ mean curvature does not imply
$\mathrm{C}^{1,1}$ regularity
in general dimensions. Indeed, for $n=3$ and $s\in{(0,1)}$, the boundary of
\begin{align*}
    E\coleq\left\{(x,y,z)\in \B^2_s(0)\times\R: z<f(x,y)\right\}
\end{align*}
with 
\begin{align*}
    f(x,y)\coleq\begin{cases}(x^2-y^2)\sqrt{-\log(\sqrt{
    x^2+y^2}}) &\text{for } (x,y)\in \B^2_1(0)\setminus\{0\}\\
    0 &\text{for }x=y=0
    \end{cases}
\end{align*}
is $\C^{1,\alpha}$ for all $\alpha\in(0,1)$ but not $\C^{1,1}$ in $\B_s^2(0)\times\R$. Moreover, one can check that the constant-in-the-third-component extension $V$ of the outward unit normal of $E$ defines a vector field in $\mathrm{W}^{1,1}(U;\R^n)\cap \C(U;\R^n)$ with divergence in $\L^\infty(U)$, where $U\coleq\B^2_s(0)\times {(-r,r)}$, $r>\|f\|_{\C(\B_s^2(0))}$, and Proposition \ref{Prop: For counterExample} below ensures the divergence of $V$ to be a variational mean curvature of $E$ in $U$. Hence, for $n\ge3$ we cannot expect analogous $\C^{1,1}$ regularity results.

\section{Sharp counterexamples}\label{Sec: Optimality}

In this section we confirm the optimality of the Hölder exponent by
constructing the sharp counterexamples announced in the introduction.

\subsection{An explicit example in two dimensions}\label{subsec:2d-example}

In order to determine a suitable curvature for the subsequent counterexample, we make use of the following proposition, which allows to obtain a variational mean curvature in some analogy with the definition of the classical mean curvature. Slightly differing versions of this statement have been given e.g. in \cite[pp.\@ 152 sq.\@]{BarozziTamanini1988penalty}, \cite[Lemma 1.3]{GMT1993boundaries}, and \cite[Proposition 4.1]{massari1994variational}. However, we find it worth recording and proving a comparably sharp version of the statement (even though we subsequently need this only with $\Gamma=\emptyset$).

\begin{prop}[Divergence of a normal as variational mean curvature]\label{Prop: For counterExample}
Let $E\subseteq\R^n$ be an open set of finite perimeter in $U$. In addition, assume that the \textup{(}weak\textup{)} outward unit normal $\nu_E$ of\/ $E$ extends \textup{(}in the sense of\/ $V=\nu_E$ holding $\mathcal{H}^{n-1}$-a.e. on $\partial^\ast\!E\cap U$\textup{)} to a vector field $V\in\mathrm{W}^{1,1}(U\setminus\Gamma;\R^n)\cap \C(U\setminus\Gamma;\R^n)$ with a relatively closed\/ $\mathcal{H}^{n-1}$-null set $\Gamma\subseteq U$ and with $|V|\le 1$ on $U\setminus\Gamma$. Then the function $H=\mathrm{div}\,V\in\mathrm{L}^1(U)$ is a variational mean curvature of\/ $E$ in $U$.
\end{prop}

\begin{proof}
  We first assume $\Gamma=\emptyset$ and thus $V\in \mathrm{W}^{1,1}(U;\R^n)\cap \C(U;\R^n)$. Let $F\in\mathcal{M}^n$ be a set of finite perimeter in $U$ such that $E\sd  F\ssubset U$. We will prove
  \begin{equation}\label{eq:minimality-extended-normal}
    \P(F,U)-\int_{F\cap U} \mathrm{div}\,V\dx
    \geq \P(E,U)-\int_{E\cap U} \mathrm{div}\,V\dx\,.
  \end{equation}
  Since we can choose a smooth open set $\tilde{U}\subseteq \R^n$ such that $E\sd  F\ssubset \tilde{U} \ssubset U$ and since it is enough to prove the inequality with $\tilde{U}$ instead of $U$, in the sequel we directly assume that $U$ itself is smooth and bounded with $V\in\C(\overline U;\R^n)$. By the structure theorem of De Giorgi and the generalized divergence theorem on the bounded finite-perimeter set $U$, we obtain
  \[
    \P(F,U)-\int_{F\cap U}\mathrm{div}\,V\dx
    = \int_{\partial^\ast\!F\cap U}1\,\d \mathcal{H}^{n-1}
      - \int_{\partial^\ast(F\cap U)}V\cdot\nu_{F\cap U}\,\d \mathcal{H}^{n-1}\,.
  \]
  Since $F$ differs from $E$ only away from $\partial U$, by locality we can split the last term into a portion inside the open set $U$ and a boundary portion on $\partial U$. In fact, $\nu_{F\cap U}=\nu_F$ holds on $\partial^\ast(F\cap U)\cap U=\partial^\ast\!F\cap U$, and $\nu_{F\cap U}=\nu_{E\cap U}$ holds on $\partial^\ast(F\cap U)\cap\partial U=\partial^\ast(E\cap U)\cap\partial U$. Thus, we get
  \begin{align*}
    &\P(F,U)-\int_{F\cap U}\mathrm{div}\,V\dx\\
    = &\int_{\partial^\ast\!F\cap U}\big(1-V\cdot\nu_F\big)\,\d\mathcal{H}^{n-1}
      - \int_{\partial^\ast(E\cap U)\cap\partial U}V\cdot\nu_{E\cap U}\,\d\mathcal{H}^{n-1}\,.
  \end{align*}
  Here, only the first term $I_F \coleq \int_{\partial^\ast\!F\cap U}\big(1-V\cdot\nu_{F}\big)\,\d\mathcal{H}^{n-1}$ on the right-hand side depends on $F$ and satisfies $I_F \geq 0$ by the Cauchy-Schwarz inequality and the assumption $|V|\le1$. Moreover, the same rewriting applies with $E$ instead of $F$, and in view of the $\mathcal{H}^{n-1}$-a.e. equality $V=\nu_E$ on $\partial^\ast\!E\cap U$ we have $I_F \geq 0 = I_E$. Therefore, altogether we infer that \eqref{eq:minimality-extended-normal} holds as required.

  \smallskip
  
  Now we turn to $V\in\mathrm{W}^{1,1}(U\setminus\Gamma;\R^n)\cap\C(U\setminus\Gamma;\R^n)$ with non-empty $\Gamma$. Once more let $F\in\mathcal{M}^n$ be a set of finite perimeter in $U$ such that $E\sd F\ssubset U$. Similar as above it suffices to establish \eqref{eq:minimality-extended-normal} where we can and do assume that $U$ is bounded and even the closure $\overline\Gamma$ of $\Gamma\subseteq U$ in $\R^n$ satisfies $\mathcal{H}^{n-1}(\overline\Gamma)=0$. Taking into account the definition of the Hausdorff measure and the compactness of $\overline\Gamma$, we then exploit $\mathcal{H}^{n-1}(\overline\Gamma)=0$ to cover $\overline\Gamma$ by finite unions of small balls which successively yield bounded open sets $N_k\subseteq\R^n$ with $N_{k+1}\subseteq N_k$ and $\bigcap_{k=1}^\infty N_k=\overline\Gamma$ such that $\P(N_k)<\frac1k$ for $k\in\mathds{N}$. We also fix slightly smaller open sets $M_k\ssubset N_k$ with $M_{k+1}\subseteq M_k$ and $\bigcap_{k=1}^\infty M_k=\Gamma$, and we introduce $U_k\coleq U\setminus\overline{M_k}$ with $U_k\subseteq U_{k+1}$ and $\bigcup_{k=1}^\infty U_k=U\setminus\Gamma$. Then, we have $E\sd  \tilde{F}_k\ssubset U_k$ for $\tilde{F}_k\coleq(F\setminus N_k)\cup(E\cap N_k)$, and from the preceding part of the reasoning we obtain $\mathrm{div}\,V\in\mathds{H}^1(E,U_k)$ for all $k\in\mathds{N}$. Moreover, since De Giorgi's structure theorem guarantees $\P(E,\Gamma)=0$ and we clearly have $|\Gamma|=0$, we find
  \begin{align}\label{eq:estimate-E,U-to-F,U_k}
       &\P(E,U)- \int_{E\cap U} \mathrm{div}\,V\dx\notag\\
       = &\lim_{k\to\infty}\bigg(\P(E,U_k) -\int_{E\cap U_k} \mathrm{div}\,V\dx\bigg)\notag\\
      \le &\limsup_{k\to\infty}\bigg(\P(\tilde{F}_k,U_k) -\int_{\tilde{F}_k \cap U_k} \mathrm{div}\,V\dx\bigg)\notag\\
      \le &\limsup_{k\to\infty}\bigg(\P(F\setminus N_k, U_k) +\P( E\cap N_k,U_k)  -\int_{\tilde{F}_k \cap U_k} \mathrm{div}\,V\dx\bigg)\,,
  \end{align}
  where we used Lemma \ref{Lemma: Properties of Perimeters} \ref{Perimeter inequality} in the last step. Using once more Lemma \ref{Lemma: Properties of Perimeters} \ref{Perimeter inequality} and exploiting $\P(F,\Gamma)=0$, we see
  \[
    \P(F\setminus N_k, U_k)
    \le \P(F,U_k)+\P(\R^n\setminus N_k,U_k)
    \le \P(F,U_k)+\P(N_k)
    \overset{k\to\infty}{\longrightarrow} \P(F,U)\,.
  \]
  Similarly, taking into account $\partial^\ast(E\cap N_k)\subset(\partial^\ast\!E\cap N_k)\cup\partial^\ast\! N_k$ and $\P(E,\Gamma)=0$, we get
  \[
    \P(E\cap N_k,U_k)
    \le\P(E\cap N_k,U)
    \le\P(E,N_k\cap U)+\P(N_k)
    \overset{k\to\infty}{\longrightarrow} 0\,.
  \]
  Finally, we observe
  \[
    \int_{\tilde{F}_k \cap U_k} \mathrm{div}\,V\dx
    \overset{k\to\infty}{\longrightarrow} \int_{F\cap U} \mathrm{div}\,V\dx\,.
  \]
  Using the convergences obtained on the right-hand side of \eqref{eq:estimate-E,U-to-F,U_k}, we arrive at \eqref{eq:minimality-extended-normal} and have thus proved $\mathrm{div}\,V\in \mathds{H}^1(E,U)$ in the general case.
\end{proof}

  We now turn to the counterexample in dimension $n=2$. Exploiting the
  geometry and Proposition \ref{Prop: For counterExample}, we are able to specify
  an explicit variational mean curvature with suitable integrability
  properties.

\begin{theo}[Counterexample to Massari-type regularity for $n=2$, $\alpha > \frac{p-2}{p+1}$]\label{Ex: Optimal Example}
  Let $n=2$, $\alpha\in{(0,1)}$, and consider the set of finite perimeter
  \[
    E\coleq\left\{x\in {(-1,1)}^2 : x_2<\mathrm{sgn}(x_1) |x_1|^{1+\alpha}\right\}\,, \\
  \]
  whose boundary $\partial E$ is $\C^{1,\alpha}$ in ${(-1,1)}^2$ but not $\C^{1,\beta}$ near $0$ for any $\beta >\alpha$. Then, there exists a variational mean curvature $H$ of\/ $E$ in ${(-1,1)}^2$ such that $H\in\mathrm{L}^p\big({(-1,1)}^2\big)$ for all\/ $p\in{[1,\infty)}$ with $\alpha>\frac{p-2}{p+1}$.
\end{theo}

Given $\beta\in{(0,1]}$, $p\in{[2,\infty)}$ with
$\beta>\alpha_\mathrm{opt}(2,p)=\frac{p-2}{p+1}$, the theorem applied with any
choice of $\alpha$ in between $\alpha_\mathrm{opt}(2,p)$ and $\beta$ disproves
$\C^{1,\beta}$ regularity of $E$. Hence, in case $n=2<p$ the example confirms the
optimality of our exponent $\alpha_\mathrm{opt}(2,p)=\frac{p-2}{p+1}$ in Theorem \ref{Theo: Improvement reg theo} in the
sense claimed in the introduction.

\begin{proof}
We define 
\begin{align*}
    D_1&\coleq \left\{x\in{(-1,1)}^2: |x_2|<|x_1|^{1+\alpha}\right\}\,,\\
    D_2&\coleq \left\{x\in{(-1,1)}^2 : |x_2|>|x_1|^{1+\alpha}\right\}\,,\\
    D_1^+&\coleq D_1\cap \left({(0,1)}\times \R\right)\,,\qquad D_1^-\coleq D_1\cap \left((-1,0)\times\R\right)\,,\\
    D_2^+&\coleq D_2\cap \left(\R\times {(0,1)}\right)\,,\qquad D_2^-\coleq D_2\cap \left(\R\times(-1,0)\right)\,.
\end{align*}

\begin{figure}[ht]
\captionsetup[subfigure]{labelformat=empty}
  \centering
  \subcaptionbox{The set $E$ and the subdomains $D_i^\pm$}{
  \begin{tikzpicture}[scale=2.5]

  \filldraw[ white, fill=orange!25] (-1,-1) -- (1,1)--(1,-1)--(-1,-1);
  
\filldraw[thick,white, fill=white] (0,0)--(0,1)--(1,1);

\filldraw[thick, color=orange!70, fill=white, variable=\x,domain=0:1,samples=20] 
  
  plot ({\x},
 {\x^(1.5)})
 node[right] {};

\filldraw[thick, color=orange!70, fill=orange!25, variable=\x,domain=-1:0,samples=20] 
  
  plot ({\x},
 {-(-\x)^(1.5)})
 node[right] {};
  
\draw[dashed, gray] (-1,-1) -- (-1,1)--(1,1)--(1,-1)--(-1,-1);
\filldraw[ color=orange!25] (-0.99,-0.99) -- (0,0)--(0.99,-0.99)--(-0.99,-0.99); 
\node[orange] at (0.8,0.5) {\Large$E$};
 
 \draw[dashed, color=red!70, variable=\x,domain=0:1,samples=20] 
  
  plot ({\x},
 {-\x^(1.5)})
 node[right] {};

\draw[dashed, color=red!70, variable=\x,domain=-1:0,samples=20] 
  
  plot ({\x},
 {(-\x)^(1.5)})
 node[right] {};
 
\node[gray] at (0,0.5) {\Large$D^+_2$};

\node[gray] at (-0.6,0) {\Large$D^-_1$};
 
\node[gray] at (0,-0.5) {\Large$D^-_2$};
 
\node[gray] at (0.6,0) {\Large$D^+_1$};

\end{tikzpicture}}
\qquad 
\subcaptionbox{Lines of constancy of the extension $V$}{  
\begin{tikzpicture}[scale=2.5]

  \filldraw[ white, fill=orange!10] (-1,-1) -- (1,1)--(1,-1)--(-1,-1);
  
\filldraw[thick,white, fill=white] (0,0)--(0,1)--(1,1);

\filldraw[dashed, color=orange!70, fill=white, variable=\x,domain=0:1,samples=20] 
  
  plot ({\x},
 {\x^(1.5)})
 node[right] {};

\filldraw[dashed, color=orange!70, fill=orange!10, variable=\x,domain=-1:0,samples=20] 
  
  plot ({\x},
 {-(-\x)^(1.5)})
 node[right] {};
  
\draw[dashed, gray] (-1,-1) -- (-1,1)--(1,1)--(1,-1)--(-1,-1);
\filldraw[ color=orange!10] (-0.99,-0.99) -- (0,0)--(0.99,-0.99)--(-0.99,-0.99);

 \draw[dashed, color=red!70, variable=\x,domain=0:1,samples=20] 
  
  plot ({\x},
 {-\x^(1.5)})
 node[right] {};

\draw[dashed, color=red!70, variable=\x,domain=-1:0,samples=20] 
  
  plot ({\x},
 {(-\x)^(1.5)})
 node[right] {};

\draw[dashed, color=blue] (-0.925,-0.9)--(0.925,-0.9);
 
\draw[dashed, color=blue] (-0.925,0.9)--(0.925,0.9);
  
\draw[dashed, color=blue] (-0.85,-0.8)--(0.85,-0.8);
   
\draw[dashed, color=blue] (-0.85,0.8)--(0.85,0.8);

\draw[dashed, color=blue] (-0.775,0.7)--(0.775,0.7);
 
\draw[dashed, color=blue] (-0.775,-0.7)--(0.775,-0.7); 

\draw[dashed, color=blue] (-0.7,0.6)--(0.7,0.6);

\draw[dashed, color=blue] (-0.7,-0.6)--(0.7,-0.6);

\draw[dashed, color=blue] (-0.625,0.5)--(0.625,0.5);

\draw[dashed, color=blue] (-0.625,-0.5)--(0.625,-0.5);

\draw[dashed, color=blue] (-0.535,0.4)--(0.535,0.4);

\draw[dashed, color=blue] (-0.535,-0.4)--(0.535,-0.4);

\draw[dashed, color=blue] (-0.44,0.3)--(0.44,0.3);

\draw[dashed, color=blue] (-0.44,-0.3)--(0.44,-0.3);

\draw[dashed, color=blue] (-0.325,0.2)--(0.325,0.2);

\draw[dashed, color=blue] (-0.325,-0.2)--(0.325,-0.2);

\draw[dashed, color=blue] (-0.22,0.1)--(0.22,0.1);

\draw[dashed, color=blue] (-0.22,-0.1)--(0.22,-0.1);

\draw[dashed, color=green] (-0.925,-0.875)--(-0.925,0.875);
 
\draw[dashed, color=green] (0.925,-0.875)--(0.925,0.875);
  
\draw[dashed, color=green] (-0.85,-0.775)--(-0.85,0.775);
   
\draw[dashed, color=green] (0.85,-0.775)--(0.85,0.775);

\draw[dashed, color=green] (0.775,-0.675)--(0.775,0.675);
 
\draw[dashed, color=green] (-0.775,-0.675)--(-0.775,0.675); 

\draw[dashed, color=green] (-0.7,-0.575)--(-0.7,0.575);

\draw[dashed, color=green] (0.7,-0.575)--(0.7,0.575);

\draw[dashed, color=green] (-0.625,-0.475)--(-0.625,0.475);

\draw[dashed, color=green] (0.625,-0.475)--(0.625,0.475);

\draw[dashed, color=green] (-0.535,-0.375)--(-0.535,0.375);

\draw[dashed, color=green] (0.535,-0.4)--(0.535,0.375);

\draw[dashed, color=green] (-0.44,-0.275)--(-0.44,0.275);

\draw[dashed, color=green] (0.44,-0.275)--(0.44,0.275);

\draw[dashed, color=green] (-0.325,-0.175)--(-0.325,0.175);

\draw[dashed, color=green] (0.325,-0.175)--(0.325,0.175);

\draw[dashed, color=green] (-0.22,-0.1)--(-0.22,0.075);

\draw[dashed, color=green] (0.22,-0.1)--(0.22,0.075);

\end{tikzpicture}}
  \caption{Counterexample to Massari-type regularity for $\alpha>\frac{p-2}{p+1}$}
\end{figure}
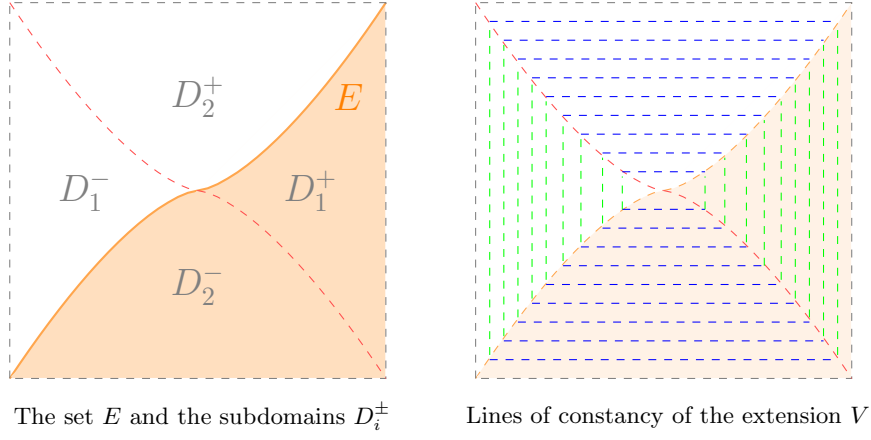

Furthermore, we compute the outward unit normal of $E$ as
\begin{align*}
    \nu_E(x) = \frac{(-(1+\alpha)|x_1|^\alpha,1)}{\sqrt{1+(1+\alpha)^2 |x_1|^{2\alpha}}}= \frac{(-(1+\alpha)|x_2|^\frac{\alpha}{1+\alpha},1)}{\sqrt{1+(1+\alpha)^2 |x_2|^{2\frac{\alpha}{1+\alpha}}}}
\end{align*}
for $(\mathrm{sgn}(x_2)|x_2|^\frac{1}{1+\alpha},x_2)=x=(x_1,\mathrm{sgn}(x_1)|x_1|^{1+\alpha})\in\partial E\cap (-1,1)^2$.
Since $(-x_1,x_2)\in\partial E\cap (-1,1)^2$ for all $x\in(\partial D_i\setminus \partial E)\cap (-1,1)^2$, $i\in\{1,2\}$, we can extend $\nu_E$ continuously to the vector field
\begin{align*}
    V(x) \coleq
    \begin{cases}
    \frac{(-(1+\alpha)|x_1|^\alpha,1)}{\sqrt{1+(1+\alpha)^2 |x_1|^{2\alpha}}} &\text{if }x\in D_1  \\
    \frac{(-(1+\alpha)|x_2|^\frac{\alpha}{1+\alpha},1)}{\sqrt{1+(1+\alpha)^2 |x_2|^{2\frac{\alpha}{1+\alpha}}}} &\text{if }x\in D_2\\
    \nu_E(x) &\text{if }x\in\partial E\\
    \nu_E(-x_1,x_2)&\text{if }x\in \partial D_1\setminus\partial E = \partial D_2\setminus \partial E.
    \end{cases}
\end{align*}
We show that $\mathrm{D}V$ exists weakly in $\L^1((-1,1)^2;\R^{2\times 2})$. Clearly, $V$ is $\C^1$ on $ D_1\cup D_2$. For $x\in D_1^+$, we calculate
\begin{align*}
  \partial_2 V(x) = 0, \qquad \partial_1 V(x) = -(1+\alpha)\alpha x_1^{\alpha-1}\frac{\left(1,(1+\alpha)x_1^\alpha\right)}{\left(1+(1+\alpha)^2x_1^{2\alpha}\right)^\frac{3}{2}}
\end{align*}
and for $x\in D_1^-$, we have $\mathrm{D}V(x)= -\mathrm{D}V(-x)$. Hence, we find $|\mathrm{D}V(x)|\le \mathrm{c}(\alpha)|x_1|^{\alpha-1}$ for $x\in D_1$ and $\mathrm{D}V\in\L^1(D_1;\R^{2\times2})$. For $x\in D_2^+$, we have
\begin{align*}
  \partial_1 V(x)=0,\qquad \partial_2 V(x)= -\alpha x_2^{\frac{\alpha}{1+\alpha}-1}\frac{\left(1, (1+\alpha)x_2^{\frac{\alpha}{1+\alpha}}\right)}{\left(1+(1+\alpha)^2x_2^{2\frac{\alpha}{\alpha+1}}\right)^\frac{3}{2}}
\end{align*}
and for $x\in D_2^-$, it holds $\mathrm{D}V(x)=-\mathrm{D}V(-x)$ by symmetry. Again, we infer $|\mathrm{D}V(x)|\le \mathrm{c}(\alpha)|x_2|^{\frac{\alpha}{1+\alpha}-1}$ for $x\in D_2$ and $\mathrm{D}V\in\L^1(D_2;\R^{2\times2})$. Since $V$ is continuous on all of ${(-1,1)}^2$ and even $\mathrm{C}^1$ away from the two regular $\mathrm{C}^1$ curves $\partial E\cap{(-1,1)}^2$ and $(\partial D_i\setminus\partial E)\cap{(-1,1)}^2$, we conclude that $V$ is in $\mathrm{W}^{1,1}((-1,1)^2;\R^2)\cap C((-1,1)^2;\R^2)$. Thus, Proposition \ref{Prop: For counterExample} with $\Gamma=\emptyset$ implies that $H\coleq\mathrm{div}\,V$ is a variational mean curvature of $E$, and we have
\begin{align*}
    \left|H(x)\right|\le \mathrm{c}(\alpha)\left(|x_1|^{\alpha-1}\mathds{1}_{D_1}(x)+|x_2|^{2\frac{\alpha}{1+\alpha}-1}\mathds{1}_{D_2}(x)\right)\,.
\end{align*}
for a.e. $x\in{(-1,1)}^2$. By symmetry and Fubini's theorem, we compute
\begin{align*}
    &\int_{(-1,1)^2}|H|^p\dx \\
    \le &4\mathrm{c}(\alpha)\int_0^1 \int_0^{x_1^{1+\alpha}}x_1^{(\alpha-1)p}\dx_2\dx_1+4\mathrm{c}(\alpha)\int_0^1 \int_0^{x_2^\frac{1}{1+\alpha}} x_2^{\left(2\frac{\alpha}{1+\alpha}-1\right)p}\dx_1\dx_2 \\
    = &4\mathrm{c}(\alpha)\int_0^1 t^{1+\alpha-(1-\alpha)p}+ t^{\frac{1}{1+\alpha}-\left(\frac{1-\alpha}{1+\alpha}\right)p}\d t.
\end{align*}
One checks that the last integral is finite if and only if $\alpha>\frac{p-2}{p+1}$. Hence, the variational mean curvature $H=\mathrm{div}\,V$ is in $\mathrm{L}^p({(-1,1)}^2)$ whenever $\alpha>\frac{p-2}{p+1}$.
\end{proof}

\subsection{A slightly less explicit example in higher dimensions}

Finally, we turn to a related example which confirms the optimality of
$\alpha_\mathrm{opt}(n,p)=\frac{p-n}{p+1}$ in arbitrary dimension $n\ge2$.
Since it is not clear to us if and how the construction of
Section \ref{subsec:2d-example}, in particular the extension of the unit normal,
generalizes to higher dimensions, we rely instead on the Barozzi construction of
the optimal variational mean curvature and estimate this curvature suitably
via Lemma \ref{lemma: Minim contain balls}.

\begin{theo}[Counterexample to Massari-type regularity for $n\ge2$, $\alpha>\frac{p-n}{p+1}$]
  Let $\alpha\in{(0,1)}$, and consider the convex set $E$ obtained as the union of
  \[
    \tilde{E}\coleq\{x\in\R^n: |\cx|^{1+\alpha}<x_n<1\}
  \]
  and the ball
  $B\coleq\B_{\sqrt{1+\frac{1}{(1+\alpha)^2}}}\big(\closure{0},1+\frac{1}{1+\alpha}\big)$
  \textup{(}which is chosen such that $\partial E$ is $\C^1$\textup{)}. Then,
  $\partial E=\partial^\ast\!E$ is $\C^{1,\alpha}$, but not $\C^{1,\beta}$ near $0$ for any
  $\beta>\alpha$, and, for every ball\/ $U\subseteq\R^n$ with $E\ssubset U$,
  there exists a variational mean curvature $H$ of\/ $E$ in $U$ such that\/
  $H\in\L^p(U)$ for all $p\in{[1,\infty)}$ with $\alpha>\frac{p-n}{p+1}$.
\end{theo}

\begin{proof}
  The set $E$ has a bounded $\C^{1,\alpha}$-boundary and thus finite perimeter
  in $U$.

  \emph{Step 1. Estimation of the curvature on $U\setminus E$.}
  Since $E$ is convex and $U$ is a ball with $E\ssubset U$, there exists $\varepsilon>0$ such that for
  each $x\in U\setminus  \closure{E}$, there exists $w\in U$ with
  $x\in \B_\varepsilon(w) \subseteq U\setminus E$
  For all
  $\lambda>\frac{n}{\varepsilon}$ and the problem of type \eqref{Minimization
  problem constant VMC} with $E$ replaced by $U\setminus E$ there,
  Lemma \ref{lemma: Minim contain balls} gives that minimizers $(U\setminus
   E)_\lambda$ necessarily contain the balls $\B_\varepsilon(w)$. For
  the variational mean curvature $H_{U\setminus  E}$ from Construction
  \ref{Construction H_E} (applied with $h_{U\setminus  E}\equiv1$ on
  $U\setminus  E$), it follows
  $0\le H_{U\setminus  E}\le \frac{n}{\varepsilon}$ on $U\setminus E$.
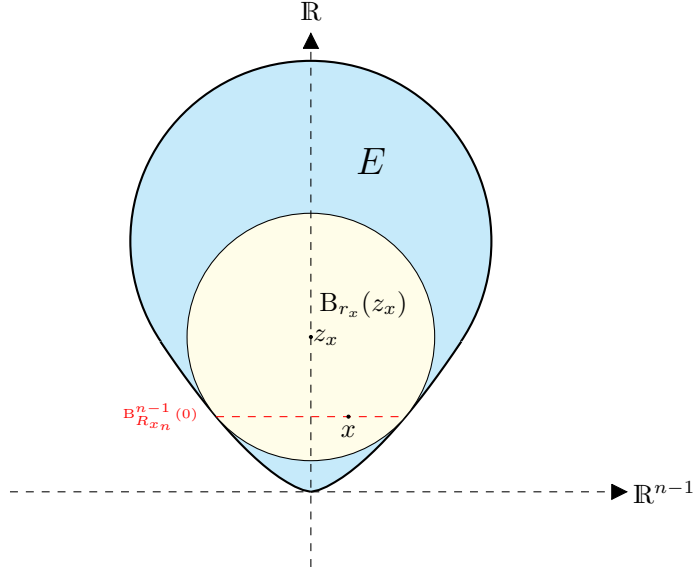
\begin{figure}[ht]
  \centering
  \begin{tikzpicture}[scale=2]
  
 \filldraw[thick, fill=cyan!20] (0,1.667) circle (1.2);

 \filldraw[color=cyan!20] (0,0)--(-1,1)--(1,1)--(0,0);

\filldraw[thick, fill=cyan!20, variable=\x,domain=0:1,samples=20] 
  
  plot ({\x},
 {\x^(1.5)});
 
 \filldraw[thick, fill=cyan!20, variable=\x,domain=-1:0,samples=20] 
  
  plot ({\x},
 {(-\x)^(1.5)});

  \filldraw[fill=yellow!10] (0,1.03) circle (0.823);

 \draw[dashed] (-2,0)--(2,0);

 \draw[dashed] (0,-0.5)--(0,3);

 \node at (0.1,1.03) {$z_x$};

 \node at (0.35,1.25) {$\B_{r_x}(z_x)$};

 \node at (0.25,0.4) {$x$};

\draw[dashed, red] (-0.63,0.5)--(0.63,0.5);

\node at (0,3.2) {$\R$};

\node at (2.35,0) {$\R^{n-1}$};

\filldraw (-0.05,2.95)--(0.05,2.95)--(0,3.05)--(-0.05,2.95);

\filldraw (2,0.05)--(2,-0.05)--(2.1,0)--(2,0.05);

\node[red] at (-1,0.5) {\tiny$\B^{n-1}_{R_{x_n}}(0)$};

 \filldraw (0.25,0.5) circle (0.01);

 \filldraw (0,1.03) circle (0.01);

\node at (0.4,2.2) {\Large $E$};
  
\end{tikzpicture}
  \caption{Counter-example to Massari-type regularity for $\alpha>\frac{p-n}{p+1}$}
\end{figure}

  \emph{Step 2. Estimation of the curvature on $E$.}
  Now, we estimate the variational mean curvature $H_E$
  from Construction \ref{Construction H_E} (with $h_{E}\equiv1$ on $E$) by
  applying Lemma \ref{lemma: Minim contain balls} for balls contained in $E$ in
  entirely the same way. Clearly, by using the ball $B$ we get
  $0\le H_E\le n\big(1+\frac{1}{(1+\alpha)^2}\big)^{-\frac{1}{2}}$ on
  $B$. For the main argument, now consider $x\in\tilde{E}$. Then,
  in view of the choice of $\tilde E$, one checks $x\in \B_{r_x}(z_x)\subseteq E$ for
\begin{align*}
    z_x&\coleq\left(\closure{0}, \frac{1}{1+\alpha}x_n^{\frac{1-\alpha}{1+\alpha}} +x_n\right)\,,\\
    r_x&\coleq \sqrt{\frac{1}{(1+\alpha)^2}x_n^{2\frac{1-\alpha}{1+\alpha}}+x_n^\frac{2}{1+\alpha}}<\sqrt{\frac{1}{(1+\alpha)^2}+1}\,,
\end{align*}
where the choices have been expressed in terms of the last component $x_n$ of $x$.
Thus, we infer
\begin{align*}
  0\le H_E(x) \le \frac{n}{r_x}\le \mathrm{c}(n,\alpha)x_n^{-\frac{1-\alpha}{1+\alpha}}
  \qquad\text{for }x\in\tilde{E}\,.
\end{align*}

\emph{Step 3. $\mathrm{L}^p$ integrability of the curvature.}
According to Remark \ref{Rem: regularity theorem of Massari},
by letting $H\coleq H_E\mathds{1}_E-H_{U\setminus  E}\mathds{1}_{U\setminus  E}$, we obtain a variational mean curvature $H$ of
$E$ in $U$, and in view of the constant bounds on $U\setminus E$
and $B$, it suffices to check the integrability of $H$ on $\tilde{E}$. We abbreviate $R_{x_n}\coleq x_n^\frac{1}{1+\alpha}$
for $x_n\ge 0$. With Fubini's theorem and $\tilde{E}_{x_n}\coleq\{\cx\in \R^{n-1}: (\cx,x_n)\in \tilde{E}\}=\B^{n-1}_{R_{x_n}}(0)$ for all $x_n\in (0,1)$, it follows
\begin{align*}
    \int_{\tilde{E}} |H|^p\dx = \int_0^1 \int_{\B^{n-1}_{R_{x_n}}(0)}|H(x)|^p\,\d\cx\dx_n\le \mathrm{c}(n,p,\alpha) \int_0^1 x_n^{\frac{n-1}{1+\alpha}-\frac{1-\alpha}{1+\alpha}p}\dx_n,
\end{align*}
and the last integral is finite if and only if
$\alpha>\frac{p-n}{p+1}$. Hence,
the variational mean curvature $H$ is in $\mathrm{L}^p(U)$ whenever
  $\alpha>\frac{p-n}{p+1}$.
\end{proof}

\begin{rem}
    In case $n=2$, the exact determination of the global variational mean curvature $H_E$ by using 
    \cite[Theorem 2.3]{leonardi2020minimizers} (compare also \cite[Theorem 3.32]{StredulinskyZiemer1997}) shows that the integrability found above is best possible. Indeed, since $E\ssubset U$ is convex, the $\L^p$ minimality from \cite[Theorem 3.2]{Barozzi1994}, Lemma \ref{Lemma: composed curvtaure} and the result in \cite[Theorem 4.2]{BarozziMassari2018} imply that $H_E\mathds{1}_E$ has the best possible integrability on $E$ even among local variational mean curvatures of $E$ in $U$. Therefore, the limit case $\mathrm{C}^{1,\alpha}$ regularity with $\alpha=\frac{p-n}{p+1}$ --- which indeed we believe does hold for $n<p<\infty$ --- at least cannot be generally ruled out by this type of example.
\end{rem}

\phantomsection


\addcontentsline{toc}{section}{References}
\begin{thebibliography}{10}

\bibitem{ambrosio2000}
L.~Ambrosio, N.~Fusco, and D.~Pallara, \emph{{Functions of Bounded Variation
  and Free Discontinuity Problems}}, Oxford University Press, 2000.

\bibitem{AmbrosioPaolini1999}
L.~Ambrosio and E.~Paolini, \emph{{Partial regularity for quasi minimizers of
  perimeter}}, Ric. Mat. \textbf{{\bf48}} (1999), 167--186.

\bibitem{Barozzi1994}
E.~Barozzi, \emph{{The curvature of a set with finite area}}, Atti Accad. Naz.
  Lincei, Cl. Sci. Fis. Mat. Nat., IX. Ser., Rend. Lincei, Mat. Appl.
  \textbf{{\bf5}} (1994), 149--159.

\bibitem{BarozziMassari2018}
E.~Barozzi and U.~Massari, \emph{{A new functional for the Calculus of
  Variations, involving the variational mean curvature of sets in $
  \mathbb{R}^{n}$}}, Manuscr. Math. \textbf{{\bf157}} (2018), 1--12.

\bibitem{BarozziTamanini1988penalty}
E.~Barozzi and I.~Tamanini, \emph{{Penalty methods for minimal surfaces with
  obstacles}}, Ann. Mat. Pura Appl., IV. Ser. \textbf{{\bf152}} (1988),
  139--157.

\bibitem{bombieri1969minimal}
E.~Bombieri, E.~{De Giorgi}, and E.~Giusti, \emph{{Minimal cones and the
  Bernstein problem}}, Invent. Math. \textbf{{\bf7}} (1969), 243–268.

\bibitem{DeGiorgi6061}
E.~{De Giorgi}, \emph{{Frontiere Orientate di Misura Minima}}, Seminario di
  Matematica, Sc. Norm. Super. Pisa., 1960--1961.

\bibitem{giaquinta1983differentiability}
M.~Giaquinta and E.~Giusti, \emph{{Differentiability of minima of
  non-differentiable functionals}}, Invent. Math. \textbf{{\bf72}} (1983),
  285--298.

\bibitem{giaquinta1984sharp}
M.~Giaquinta and E.~Giusti, \emph{{Sharp estimates for the derivatives of local
  minima of variational integrals}}, Boll. Unione Mat. Ital., VI. Ser., A
  \textbf{{\bf3}} (1984), 239--248.

\bibitem{giusti1984}
E.~Giusti, \emph{{Minimal Surfaces and Functions of Bounded Variation}},
  Birk\-häuser, 1984.

\bibitem{massari1994variational}
E.H.A. Gonzalez and U.~Massari, \emph{{Variational mean curvatures}}, Rend.
  Semin. Mat., Torino \textbf{{\bf52}} (1994), 1--28.

\bibitem{GMT1993boundaries}
E.H.A. Gonzalez, U.~Massari, and I.~Tamanini, \emph{{Boundaries of prescribed
  mean curvature}}, Atti Accad. Naz. Lincei, Cl. Sci. Fis. Mat. Nat., IX. Ser.,
  Rend. Lincei, Mat. Appl. \textbf{{\bf4}} (1993), 197--206.

\bibitem{hamburger2007optimal}
C.~Hamburger, \emph{{Optimal partial regularity of minimizers of quasiconvex
  variational integrals}}, ESAIM, Control Optim. Calc. Var. \textbf{{\bf13}}
  (2007), 639--656.

\bibitem{leonardi2020minimizers}
G.P. Leonardi and G.~Saracco, \emph{{Minimizers of the prescribed curvature
  functional in a Jordan domain with no necks}}, ESAIM, Control Optim. Calc.
  Var. \textbf{{\bf26}} (2020), 20 pages.

\bibitem{maggi2012}
F.~Maggi, \emph{{Sets of Finite Perimeter and Geometric Variational Problems}},
  Cambridge University Press, 2012.

\bibitem{massari1974esistenza}
U.~Massari, \emph{{Esistenza e regolarit{\`a} delle ipersuperfici di curvatura
  media assegnata in $\mathbb{R}^n$}}, Arch. Ration. Mech. Anal. \textbf{55}
  (1974), 357--382.

\bibitem{massari1975frontiere}
U.~Massari, \emph{{Frontiere orientate di curvatura media assegnata in $
  \mathrm{L}^p$}}, Rend. Sem. Mat. Univ. Padova \textbf{{\bf53}} (1975),
  37--52.

\bibitem{phillips1983minimization}
D.~Phillips, \emph{{A minimization problem and the regularity of solutions in
  the presence of a free boundary}}, Indiana Univ. Math. J. \textbf{{\bf32}}
  (1983), 1--17.

\bibitem{schmidt2009simple}
T.~Schmidt, \emph{{A simple partial regularity proof for minimizers of
  variational integrals}}, NoDEA, Nonlinear Differ. Equ. Appl. \textbf{{\bf16}}
  (2009), 109--129.

\bibitem{schuett2022master}
J.H.~Schütt, \emph{{Variational Mean Curvatures, Approximation of Sets of Finite
  Perimeter and Regularity of Sets with Variational Mean Curvature in
  $\mathrm{L}^p$}}, Master Thesis, Universität Hamburg, 2022.

\bibitem{StredulinskyZiemer1997}
E.~Stredulinsky and W.~Ziemer, \emph{{Area minimizing sets subject to a volume
  constraint in a convex set}}, J. Geom. Anal. \textbf{{\bf7}} (1997),
  653--677.

\bibitem{TamaniniBoundariesofCaccioppolisets}
I.~Tamanini, \emph{{Boundaries of Caccioppoli sets with H{\"o}lder-continuous
  normal vector}}, J. Reine Angew. Math. \textbf{{\bf334}} (1982), 27--39.

\bibitem{Tamanini1989Approx}
I.~Tamanini and C.~Giacomelli, \emph{{Approximation of Caccioppoli sets, with
  applications to problems in image segmentation}}, Ann. Univ. Ferrara, Nuova
  Ser., Sez. VII \textbf{{\bf35}} (1989), 187--214.

\bibitem{tamaninni1984regularity}
I.~Tamanini, \emph{{Regularity results for almost minimal oriented
  hypersurfaces in $\mathbb{R}^n$}}, Quaderni del Dipartimento di Matematica
  dell’Universit{\`a} di Lecce (1984), 92 pages.

\bibitem{yang2013optimal}
R.~Yang, \emph{{Optimal regularity and nondegeneracy of a free boundary problem
  related to the fractional Laplacian}}, Arch. Ration. Mech. Anal.
  \textbf{{\bf208}} (2013), 693--723.

\end{thebibliography}
\end{document}